\documentclass[a4paper, 10pt]{article}

\usepackage{algorithmic}
\usepackage{algorithm}
\usepackage{amsmath}
\usepackage{amsfonts}
\usepackage{amssymb}
\usepackage{amsthm}
\usepackage{tikz}
\usepackage{pgf}
\usepackage{pgfplots,  pgfplotstable}
\pgfplotsset{compat=1.5}
\usepgfplotslibrary{external}
\usepgfplotslibrary{statistics}
\usetikzlibrary{patterns}
\usepackage[utf8]{inputenc}
\usetikzlibrary{shapes,decorations}
\usepackage{siunitx} 
\usepackage{booktabs}

\usepackage{setspace}
\usepackage[displaymath, mathlines]{lineno}
\usepackage{graphicx}

\newtheorem{lemma}{Lemma}

\usepackage{comment}
\usepackage[hidelinks]{hyperref} 
\usepackage[a4paper,margin=2.54cm]{geometry}

\newtheorem*{lemma_mono}{Lemma~\ref{lemma_mono}}
\newtheorem*{lemma_chi}{Lemma~\ref{lemma_chi}}
\newtheorem*{lemma_xn}{Lemma~\ref{lemma_xn}}

\title{A fast algorithm for quadratic resource allocation problems with nested constraints}
\author{Martijn H. H. Schoot Uiterkamp, Johann L. Hurink, Marco E. T. Gerards \\ University of Twente, Enschede, the Netherlands}

\usepackage{multicol}
\usepackage{subcaption}

\begin{document}
\maketitle

\newcommand{\seq}{$\text{ALG}_{\text{seq}}$}
\newcommand{\infe}{$\text{ALG}_{\text{inf}}$}
\newcommand{\dec}{$\text{ALG}_{\text{dec}}$}
\begin{abstract}
We study the quadratic resource allocation problem and its variant with lower and upper constraints on nested sums of variables. This problem occurs in many applications, in particular battery scheduling within decentralized energy management (DEM) for smart grids. We present an algorithm for this problem that runs in $O(n \log n)$ time and, in contrast to existing algorithms for this problem, achieves this time complexity using relatively simple and easy-to-implement subroutines and data structures. This makes our algorithm very attractive for real-life adaptation and implementation. Numerical comparisons of our algorithm with a subroutine for battery scheduling within an existing tool for DEM research indicates that our algorithm significantly reduces the overall execution time of the DEM system, especially when the battery is expected to be completely full or empty multiple times in the optimal schedule. Moreover, computational experiments with synthetic data show that our algorithm outperforms the currently most efficient algorithm by more than one order of magnitude. In particular, our algorithm is able to solves all considered instances with up to one million variables in less than 17 seconds on a personal computer.
\end{abstract}

\section{Introduction}

\subsection{Resource allocation problems and energy management}

The resource allocation problem is a classical  and well-researched problem in the optimization and operations research literature. The objective of the resource allocation problem is to divide a fixed amount of resource (e.g., time, money, energy) over a set of activities while minimizing a given cost function (or maximizing a given utility function). In the most studied version of this problem, the cost functions are quadratic, which leads to the following formulation of the so-called quadratic resource allocation problem (QRAP):
\begin{align}
\text{QRAP:} \ \min_{x \in \mathbb{R}^n} \ & \sum_{i \in \mathcal{N}} \frac{1}{2} \frac{x_i^2}{a_i} \nonumber\\
\text{s.t. } & \sum_{i \in \mathcal{N}} x_i = R, \label{eq_QRAP_02}\\
& l_i \leq x_i \leq u_i, \quad i \in \mathcal{N}, \nonumber
\end{align}
where $a \in \mathbb{R}^n_{>0}$, $R \in \mathbb{R}$, $l,u \in \mathbb{R}^n$, and $\mathcal{N} := \lbrace 1,\ldots, n \rbrace$. The problem QRAP has been studied extensively over the last decades due to its wide applicability in, among others, engineering, finance, and machine learning (see also the surveys in \cite{Patriksson2008, Patriksson2015}). As a consequence, many efficient algorithms have been developed for this problem and its generalizations.

In this article, we study an extension of QRAP, namely the QRAP with lower and upper constraints on nested sums of variables (QRAP-NC). This problem can be formulated as follows:
\begin{align}
\text{QRAP-NC:}  \ \min_{x \in \mathbb{R}^n} \ & \sum_{i \in \mathcal{N}} \frac{1}{2} \frac{x_i^2}{a_i} \nonumber \\
\text{s.t. } & L_j \leq \sum_{i \in \mathcal{N}_j} x_i \leq U_j, \quad j \in \mathcal{N}_{n-1}, \nonumber \\
& \sum_{i \in \mathcal{N}} x_i = R, \nonumber \\
& l_i \leq x_i \leq u_i, \quad i \in \mathcal{N}, \label{RAP_NC_04}
\end{align}
where $\mathcal{N}_j := \lbrace 1,\ldots, j \rbrace$ for $j \in \mathcal{N} \backslash \lbrace n \rbrace$, $L,U \in \mathbb{R}^{n-1}$, and we define $L_n = U_n = R$ for convenience. Note that if $L_j = U_j$ for some $j \in \mathcal{N}_{n-1}$, we may split up the problem QRAP-NC into two smaller instances of QRAP-NC that involve the variables $x_1,\ldots,x_j$ and $x_{j+1},\ldots,x_n$ respectively. Thus, we assume without loss of generality that $L_j < U_j$ for all $j \in \mathcal{N}_{n-1}$. Moreover, we may assume that $L_1 = l_1$, $U_1 = u_1$, and $L_j \geq L_{j-1} + l_j$ and $U_j \leq U_{j-1} +  u_j$ for $j \in \mathcal{N}_{n-1} \backslash \lbrace 1 \rbrace$.

The problem QRAP-NC has numerous applications in, among others, machine learning, telecommunications, and speed optimization problems (see also the overviews in \cite{Akhil2018,Vidal2018}). Our particular motivation for studying QRAP-NC is its application in decentralized energy management (DEM) for smart distribution grids. In DEM, the goal is to optimize the joint energy consumption of multiple devices within, e.g., a neighborhood. In a DEM system, devices optimize their own consumption locally but this local optimization is coordinated to obtain certain global objectives (hence the term ``decentralized''). In the context of DEM, we are interested in optimization of storage devices such as electrical batters and heat buffers. Energy storage devices plays an important role in DEM systems since they are quite flexible in their energy usage and are thus suitable to compensate for peak consumption or production of energy in the distribution grid (see, e.g., \cite{Roberts2011,Lund2016,Zame2018}).

One important example of a device-level optimization problem within DEM is the scheduling of a battery within a neighborhood. We consider the situation where the charging and discharging of the battery has to be scheduled over a set $\mathcal{N}$ of equidistant time intervals, each of length $\Delta t$. Given the power profile $p := (p_i)_{i \in \mathcal{N}}$ of the neighborhood, the goal is to determine for each time interval $i \in \mathcal{N}$ the charging power $x_i$ of the battery during this interval so that the combined battery and neighborhood profile is flattened as much as possible. Aiming for this goal reduces the stress put on the grid and the risk of blackouts. The (physical) restrictions of the battery are given by a minimum and maximum charging rate $X_{\min}$ and $X_{\max}$ and a capacity $C$. Given the amount of energy present in the battery (the state-of-charge (SoC)) at the start and end of the scheduling horizon, denoted by $S_{\text{start}}$ and $S_{\text{end}}$, we can formulate the resulting device-level optimization problem as follows (see also \cite{vdKlauw2017}):
\begin{align*}
\text{BATTERY:}  \ \min_{x \in \mathbb{R}^n} \ & \sum_{i \in \mathcal{N}} (p_i + x_i)^2 \\
\text{s.t. } & 0 \leq S_{\text{start}} + \Delta t \sum_{i \in \mathcal{N}_j} x_i \leq D, \quad j \in \mathcal{N}_{n-1}, \\
& S_{\text{start}} + \Delta t \sum_{i \in \mathcal{N}} x_i = S_{\text{end}}, \\
& X_{\min} \leq x_i \leq X_{\max}, \quad i \in \mathcal{N}.
\end{align*}
Note that this is an instance of QRAP-NC by applying the variable transform $y := p + x$.

An important feature within the DEM paradigm is that device-level problems have to be solved locally. This means that the corresponding device-level optimization algorithms are executed on embedded systems with limited computational power (see, e.g., \cite{Beaudin2015}) that are located within, e.g., households. Since these algorithms are called multiple times with the DEM system as a subroutine, it is important that these algorithms are very efficient. Therefore, efficient and tailored device-level optimization algorithms are crucial ingredients for the real-life implementation of DEM systems. In particular, for the optimization of storage devices, this means that fast and tailored algorithms to solve QRAP-NC are crucial. For more background on DEM, we refer to \cite{Siano2014, Esther2016}.

\subsection{Background and contribution}

There is a rich literature on solution approaches for QRAP-NC with only \emph{upper} nested constraints on sums of variables,  i.e., only nested constraints of the form $\sum_{i \in \mathcal{N}_j} x_i \leq U_j$, $j \in \mathcal{N}_{n-1}$ are given. This case has been studied mainly in the context of convex optimization over submodular constraints (see, e.g., \cite{Hochbaum1994,Hochbaum1995,Vidal2016}). However, the literature on the general case of QRAP-NC is limited. The authors in \cite{vdKlauw2017} propose an infeasibility-guided divide-and-conquer algorithm, to which we shall refer in this article as \infe. This algorithm solves a relaxation of the problem where the nested constraints are ignored and, subsequently, splits up the problem into two smaller instances of QRAP-NC at the variable for which the lower or upper nested constraint is violated most in the solution to the relaxation. The worst-case time complexity of this algorithm is $O(n^2)$. Furthermore, \cite{Vidal2018} proposes a decomposition-based algorithm, hereafter referred to as \dec, that solves QRAP-NC in $O(n \log n)$ time. This algorithm decomposes QRAP-NC into a hierarchy of QRAP subproblems whose single-variable bounds are optimal solutions to QRAP subproblems further down in the hierarchy. Currently, this is the most efficient algorithm for QRAP-NC.

As mentioned before, we are interested in algorithms for QRAP-NC that are fast in practice. Although the decomposition-based algorithm \dec\ has a good worst-case time complexity, we observe several disadvantages of this approach that may make it less favorable in practice than its worst-case time complexity suggests:
\begin{enumerate}
\item
Each level of recursion within \dec\ solves a series of instances of QRAP whose parameters are determined by optimal solutions to multiple instances of QRAP on earlier levels. Since each instance is solved from scratch, much time is spent on initializing the subproblems.

\item
\dec\ achieves for each level of recursion an $O(n)$ time complexity by solving the QRAP subproblems using an $O(n)$ time algorithm such as the ones in \cite{Kiwiel2008a}. These $O(n)$ time algorithms repeatedly call linear-time algorithms such as \cite{Blum1973} to find the median of a set. However, these median-find algorithms are relatively slow in practice due to a big constant factor in their complexity \cite{Blum1973}. Moreover, they are significantly more difficult to implement than simple sorting or sampling-based strategies \cite{Kiwiel2005,Alexandrescu2017}.
\end{enumerate}

To alleviate these issues, we propose in this article a new algorithm for QRAP-NC, called \seq, which has the same time complexity as \dec, namely $O(n \log n)$, but in contrast requires only relatively simple and fast subroutines to attain this complexity. As a consequence, this algorithm is both faster in practice and easier to implement than \dec. These are generally more important criteria for the actual adaptation of a given algorithm than the polynomial worst-case time complexity \cite{Muller-Hannemann2010}. Our algorithm builds upon the monotonicity results for QRAP-NC derived in \cite{Vidal2018} and solves a sequence of QRAP subproblems that have a sequential nested structure rather than the divide-and-conquer structure of both \dec\ and \infe. More precisely, for each $j \in \mathcal{N}$, the $j^{\text{th}}$ subproblem involves only the first $j$ variables $x_1, \ldots, x_j$. As a consequence, our approach can solve its first $j$ subproblems without any knowledge on the parameters involving indices higher than~$j$, whereas both \infe\ and \dec\ require all problem parameters to be known a priori. This makes our algorithm particularly useful in situations where problem parameters arrive over time. This is, e.g., the case when each variable denotes a decision for a specific time slot and all parameters related to this time slot become available only during or at the start of this time slot. Moreover, due to the nested structure, each input and bookkeeping parameter is accessed within a relatively small time period instead of frequently throughout the entire course of the algorithm. This is beneficial for caching since this increases the number of times a value can be accessed quickly from a cache instead of relatively slowly from the main memory.

We attain the $O(n \log n)$ complexity using an efficient implementation of double-ended priority queues \cite{Knuth1998,Brass2008} for several bookkeeping parameters. This data type supports insertion of arbitrary elements and finding and deletion of minimum and maximum elements in at most $O(\log n)$ time. Our approach requires $O(n)$ of such operations, which leads to an overall time complexity of $O(n \log n)$. Double-ended priority queues can be implemented using specialized data structures such as min-max heaps \cite{Atkinson1986} or by a simple coupling of a standard min-heap and max-heap (see also \cite{Brass2008}). The latter heaps are one of the most basic data structures and many efficient implementation exist for different programming languages \cite{Brodal2013}. Thus, we can achieve the time complexity of $O(n \log n)$ using relatively simple data structures, as opposed to \dec, where a more involved implementation of a linear-time median algorithm is required.

Our algorithm for QRAP-NC also leads to efficient and fast algorithms for instances of QRAP-NC where we replace each term $\frac{1}{2} \frac{x_i^2}{a_i}$ by $ a_i f(\frac{x_i}{a_i})$ for each $i \in \mathcal{N}$ with a given convex function $f$. Such a structure is present in many applications considered in the literature, in particular in most of the applications surveyed or evaluated in \cite{Akhil2018,Vidal2018}. We obtain such efficient algorithms by a reduction result in \cite{SchootUiterkamp2020a}, which states that any optimal solution to an instance of QRAP-NC is also optimal for this instance when we take as objective function $\sum_{i \in \mathcal{N}} a_i f(\frac{x_i}{a_i})$. As a consequence, our algorithm solves also such problems in $O(n \log n)$ time. This leads to faster algorithms for a wide range of practical problems, including the vessel speed optimization problem \cite{Norstad2011,Hvattum2013} and processor scheduling with agreeable deadlines \cite{Huang2009,Gerards2014}.

We evaluate the performance of our algorithm \seq\ and compare it to the state-of-the-art algorithms \infe\ and \dec. For this evaluation, we use both synthetic instances and realistic instances of the battery scheduling problem BATTERY using real power consumption data as input. With regard to the realistic instances, we compare our approach to a tailored implementation of \infe\ using DEMKit, an existing simulation tool for DEM research \cite{Hoogsteen2019}. Within DEMKit, the battery scheduling problem is used as a subroutine within a distributed optimization framework that coordinates the energy consumption of multiple devices \cite{Gerards2015}. Our results indicate that the number of tight nested constraints in an optimal solution greatly influences which algorithm is faster for a given problem instance. In particular, \seq\ is on average faster than \infe, except when the percentage of tight nested constraints is relatively low (less than 2\%). Moreover, the execution time of \seq\ is more stable than that of \infe, which makes our algorithm more suitable for use in DEM systems that employ a high level of parallelism (see, e.g., \cite{Hoogsteen2018}). With regard to the synthetic instances, we study the scalability of \seq, \infe, and \dec. Our results indicate that both our algorithm \seq\ and \infe\ are at least one order of magnitude faster than \dec\ and that \seq\ is on average almost twice as fast as \infe. In particular, \seq\ solves instances with up to one million variables in less than 17 seconds. 

The outline of the remainder of this article is as follows. In Section~\ref{sec_BP_search}, we present a simple procedure to solve QRAP, which forms an important ingredient for our eventual approach for solving QRAP-NC. In Section~\ref{sec_seq}, we present an initial sequential algorithm \seq\ for solving QRAP-NC with an $O(n^2)$ worst-case time complexity. Based on this algorithm, we derive in Section~\ref{sec_quad} an $O(n \log n)$ time algorithm for this problem. In Section~\ref{sec_eval}, we evaluate the performance of this algorithm and compare it to the state-of-the-art. Finally, we provide our conclusions in Section~\ref{sec_concl}.

\section{A breakpoint search algorithm for QRAP}
\label{sec_BP_search}
In this section, we discuss a simple approach to solve QRAP that belongs to the class of so-called \emph{breakpoint search} methods \cite{Kiwiel2008a,Patriksson2015} that structurally search for the optimal Lagrange multiplier corresponding to the resource constraint~(\ref{eq_QRAP_02}). This approach forms an important ingredient of our $O(n \log n)$ time algorithm for QRAP-NC in Section~\ref{sec_quad}.

We start by considering the Lagrangian relaxation of QRAP:
\begin{align*}
\text{QRAP}[\delta]  : \ \min_{x \in \mathbb{R}^n} \ & \sum_{i \in \mathcal{N}} \left( \frac{1}{2} \frac{x_i^2}{a_i} -\delta x_i \right) \\
\text{s.t. } &  l_i \leq x_i \leq u_i, \quad i \in \mathcal{N},
\end{align*}
where $\delta \in \mathbb{R}$ is the Lagrange multiplier corresponding to the resource constraint~(\ref{eq_QRAP_02}). We denote the optimal solution to this problem by ${x} [\delta] := ({x}_i[\delta])_{i \in \mathcal{N}}$. Since the objective function of this problem is separable, the optimal solution to QRAP$[\delta]$ is given by
\begin{equation}
{x}_i [\delta] = \begin{cases}
l_i & \text{if } \delta < \frac{l_i}{a_i}, \\
a_i \delta & \text{if } \frac{l_i}{a_i} \leq \delta < \frac{u_i}{a_i}, \\
u_i & \text{if } \frac{u_i}{a_i} \leq \delta.
\end{cases}
\label{eq_opt_QRA}
\end{equation}
Observe that ${x}_i[\delta]$ is a continuous piecewise linear non-decreasing function of $\delta$. More precisely, ${x}_i[\delta]$ is constant for $\delta \leq \frac{l_i}{a_i}$, linear with slope $a_i$ for $\delta \in [\frac{l_i}{a_i} , \frac{u_i}{a_i} ]$, and again constant for $\delta \geq \frac{u_i}{a_i}$ (see also Figure~\ref{plot_BP}). For each $i \in \mathcal{N}$, we call the points where ${x}_i[\delta]$ has ``kinks'', i.e., where ${x}_i[\delta]$ is non-differentiable, the \emph{breakpoints} of ${x}_i[\delta]$. We denote these breakpoints for $i \in \mathcal{N}$ by $\alpha_i$ and $\beta_i$ respectively, i.e., $\alpha_i := \frac{l_i}{a_i}$ and $\beta_i := \frac{u_i}{a_i}$, where we refer to $\alpha_i$ as the \emph{lower} breakpoint of ${x}_i[\delta]$ and to $\beta_i$ as the \emph{upper} breakpoint of ${x}_i[\delta]$. We denote the multiset of lower breakpoints by $\mathcal{A} := \lbrace \alpha_i \ | \ i \in \mathcal{N} \rbrace$ and the multiset of upper breakpoints by $\mathcal{B} := \lbrace \beta_i \ | \ i \in \mathcal{N} \rbrace$. The reason for defining $\mathcal{A}$ and $\mathcal{B}$ as \emph{multi}sets is so that we can readily associate each breakpoint value in the set with one index in $\mathcal{N}$.

\begin{figure}[ht!]
\centering
\includegraphics{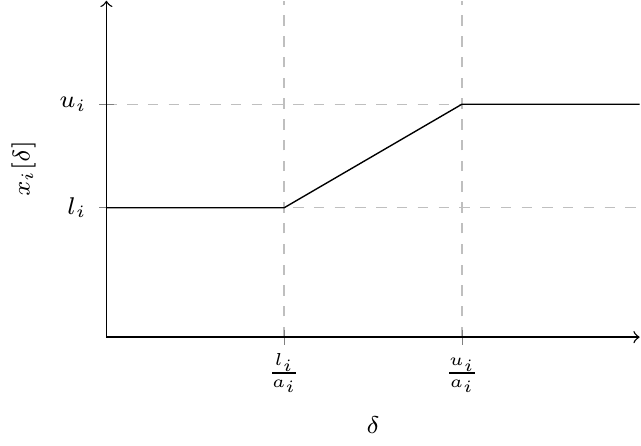}
\caption{The function ${x}_i[\delta]$ for a given $i \in \mathcal{N}$. The slope of the line segment for $\delta \in [\frac{l_i}{a_i},\frac{u_i}{a_i}]$ is $a_i$.}
\label{plot_BP}
\end{figure}

Note that also the sum ${z}[\delta] := \sum_{i \in \mathcal{N}} {x}_i [\delta]$ is continuous, piecewise linear, and non-decreasing. Moreover, it has $2n$ breakpoints, namely those of all terms ${x}_i[\delta]$. Thus, the multiset of breakpoints of ${z}[\delta]$ is given by $\mathcal{A} \cup \mathcal{B}$. Feasibility of the original problem QRAP implies that there exists a value $\bar{\delta}$ for the Lagrange multiplier $\delta$ such that ${z}[\bar{\delta}] = R$, meaning that ${x}[\bar{\delta}]$ is optimal not only for QRAP$(\bar{\delta})$ but also for the original problem QRAP. Note that this multiplier is not necessarily unique: in general, there may exist an interval $I \subset \mathbb{R}$ such that $\delta \in I$ implies $z[\delta] = R$.

Our approach to find the value $\bar{\delta}$ consists two steps. First, we aim to find two consecutive breakpoints $\delta_1$ and $\delta_2$ such that $\delta_1 \leq \bar{\delta} \leq \delta_2$. Since ${z}$ is non-decreasing, this is equivalent to finding two consecutive breakpoints $\delta_1$ and $\delta_2$ such that ${z}[\delta_1] \leq R \leq {z}[\delta_2]$. For this, we may consider all breakpoints in $\mathcal{A} \cup \mathcal{B}$ in non-decreasing order until we have found the first, i.e., smallest, breakpoint $\delta$ such that $\bar{\delta} < \delta$. In detail, for each candidate breakpoint $\delta$, we compute ${z}[\delta]$ and if ${z}[\delta] > R$, we set $\delta_2:= \delta$ and $\delta_1$ as the previously considered breakpoint. To compute ${z}[\delta]$ efficiently, we keep track of the sums 
\begin{equation*}
P(\delta) := \sum_{i: \ \delta < \frac{l_i}{a_i}} l_i + \sum_{i: \ \delta \geq \frac{u_i}{a_i}} u_i,
\quad
\quad
Q(\delta) := \sum_{i: \ \frac{l_i}{a_i} \leq \delta < \frac{u_i}{a_i}} a_i
\end{equation*}
and update these values each time a new breakpoint has been considered (see Table~\ref{tab_rel_01}).

\begin{table}[ht!]
\centering
\begin{tabular}{r | l l}
\toprule
Type of $\delta$ & Update $P(\delta)$ & Update $Q(\delta)$ \\
\midrule
$\delta \equiv \alpha_i$ & $P(\delta) - l_i$ & $Q(\delta) + a_i$ \\
$\delta \equiv \beta_i$ & $P(\delta) + u_i$ & $Q(\delta) - a_i$ \\
\bottomrule
\end{tabular}
\caption{Updating the bookkeeping sums $P(\delta)$ and $Q(\delta)$ when searching the breakpoints in non-decreasing order.}
\label{tab_rel_01}
\end{table}

In a second step, given the consecutive breakpoints $\delta_1$ and $\delta_2$ with $\bar{\delta} \in [\delta_1,\delta_2]$, we determine $\bar{\delta}$ and ${x}[\bar{\delta}]$. Note that, since ${x}[\delta]$ is non-decreasing, we have for each $i \in \mathcal{N}$:
\begin{itemize}
\item
${x}_i[\bar{\delta}] = l_i$ if and only if ${x}_i(\delta_2) = l_i $, and
\item
${x}_i[\bar{\delta}] = u_i$ if and only if ${x}_i(\delta_1) = u_i$.
\end{itemize}
Thus, given $\delta_1$ and $\delta_2$, we know whether a given variable ${x}_i[\bar{\delta}]$ equals its lower bound $l_i$, its upper bound $u_i$, or is strictly in between these bounds. To find ${x}_i[\bar{\delta}]$ for those variables that are strictly in between their bounds, note that, by definition of ${x}[\delta]$,
\begin{equation*}
R = {z}[\bar{\delta}]
= \sum_{i: \ {x}_i[\bar{\delta}] = l_i} l_i
+ \sum_{i: \ l_i < {x}_i[\bar{\delta}] < u_i} a_i \bar{\delta}
+ \sum_{i: \ {x}_i[\bar{\delta}] = u_i} u_i.
\end{equation*}
It follows that
\begin{equation*}
\bar{\delta} = \frac{R - \sum_{i: \ {x}_i[\bar{\delta}] = l_i} l_i - \sum_{i: \ {x}_i[\bar{\delta}] = u_i} u_i}{\sum_{i: \ l_i < {x}_i[\bar{\delta}] < u_i} a_i },
\end{equation*}
from which we can directly compute ${x}_i[\bar{\delta}]$ by ${x}_i[\bar{\delta}] = a_i \bar{\delta}$.

Algorithm~\ref{alg_QRAP_01} summarizes the sketched approach. To efficiently compute the minimum breakpoint $\delta_k$, we can implement the multisets $\mathcal{A}$ and $\mathcal{B}$ as sorted lists. As a consequence, each iteration of the algorithm takes $O(1)$ time. Since the maximum number of iterations is $2n$ (one for each breakpoint), the overall complexity of this approach is $O(n \log n)$ due to the initial sorting of the breakpoints. If this sorting is given (for example if the breakpoints have already been sorted in a previous run of the algorithm), the time complexity of the algorithm reduces to $O(n)$. 

\begin{algorithm}[ht!]
\caption{An $O(n \log n)$ time algorithm for QRAP.}
\label{alg_QRAP_01}
\begin{algorithmic}[1]
\STATE{\textbf{Input:} Parameters $a \in \mathbb{R}^n_{>0}$, $R \in \mathbb{R}$, and $l,u \in \mathbb{R}^n$}
\STATE{\textbf{Output:} Optimal solution ${x}$ to QRAP}
\STATE{Compute the breakpoint multisets $\mathcal{A}$ and $\mathcal{B}$}
\STATE{Initialize $P := \sum_{i=1}^n l_i$; $Q := 0$}
\REPEAT
\STATE{Determine smallest breakpoint $\delta_i := \min(\mathcal{A} \cup \mathcal{B})$}
\IF{$P+Q \delta_i  = C$}
\STATE{$\delta_i = \bar{\delta}$; compute ${x}[\bar{\delta}]$ using Equation~(\ref{eq_opt_QRA})}
\RETURN
\ELSIF{$P + Q\delta_i  > C$}
\STATE{($\bar{\delta} < \delta_i$): $\bar{\delta} = \frac{C - P}{Q}$; compute ${x}[\bar{\delta}]$ using Equation~(\ref{eq_opt_QRA})}
\RETURN
\ELSE
\IF{$\delta_i$ is lower breakpoint ($\delta_i = \alpha_i$)}
\STATE{$P := P - l_i$; $Q := Q + a_i$}
\STATE{$\mathcal{A} := \mathcal{A} \backslash \lbrace \alpha_i \rbrace$}
\ELSE
\STATE{$P := P + u_i$; $Q := Q - a_i$}
\STATE{$\mathcal{B} := \mathcal{B} \backslash \lbrace \beta_i \rbrace$}
\ENDIF
\ENDIF
\UNTIL{multiplier $\bar{\delta}$ has been found}
\RETURN{Optimal solution $\bar{x} := {x}[\bar{\delta}]$}
\end{algorithmic}
\end{algorithm}

We conclude this subsection with two observations that are crucial for the efficiency of our algorithm for QRAP-NC presented in the following section:
\begin{enumerate}
\item
Instead of searching the breakpoints in non-decreasing order, we may also search them in \emph{non-increasing order} and continue the search until we find the first, i.e., largest breakpoint $\delta_1$ such that $\delta_1 < \bar{\delta}$.
\item
Solving two instances of QRAP that differ only in the value of~$R$ in the resource constraint~(\ref{eq_QRAP_02}) can be done simultaneously in one run of Algorithm~\ref{alg_QRAP_01}. This is because the multisets of the breakpoints for these two instances of QRAP are the same. Thus, we can modify Algorithm~\ref{alg_QRAP_01} such that it continues the breakpoint search after the optimal multiplier for the smallest of the given values of~$R$ has been found. Note that, essentially, the optimal multiplier for a given value~$R$ serves as the starting candidate for the optimal multiplier for instances with a higher value of~$R$. This is in fact one of the two crucial observations for our approach for solving the QRAP subproblems, which we discuss further in Section~\ref{sec_opt_mult}.
\end{enumerate}

\section{An initial sequential algorithm for QRAP-NC}
\label{sec_seq}

In this section, we present our initial sequential algorithm for the problem QRAP-NC. This algorithm solves the problem as a sequence of $2n - 1$ instances of QRAP whose single-variable bounds~(\ref{RAP_NC_04}) are optimal solutions to previous QRAP subproblems. For this, we consider a sequence of restricted subproblems where we take into account only a subset of the variables. More precisely, we define for each $j \in \mathcal{N}$ and $C \in \mathbb{R}$ the following subproblem:
\begin{align}
\text{QRAP-NC}^j(C) \ : \ \min_{x \in \mathbb{R}^{j}} \ & \sum_{i \in \mathcal{N}_j} \frac{1}{2} \frac{x_i^2}{a_i} \nonumber \\
\text{s.t. } & \sum_{i \in \mathcal{N}_j} x_i = C, \label{sub_02} \\
& L_k \leq \sum_{i \in \mathcal{N}_k} x_i \leq U_k, \quad k \in \mathcal{N}_{j-1}, \label{sub_03}\\
&l_i \leq x_i \leq u_i, \quad i \in \mathcal{N}_j . \nonumber
\end{align}
Throughout this article, we denote the optimal solution to this subproblem by ${x}^{j}(C) := ({x}^{j}_i(C))_{i \in \mathcal{N}_j}$, where we use the brackets $( \cdot )$ instead of $[\cdot]$ to emphasize the distinction of this solution from an optimal solution ${x}[\delta]$ of the Lagrangian relaxation QRAP$(\delta)$ of QRAP. Note that this optimal solution is unique since the objective function of the corresponding problem is \emph{strictly} convex and all constraints are linear. Moreover, observe that the $n^{\text{th}}$ subproblem QRAP-NC$^n(R)$ is equal to the original problem QRAP-NC.

The key ingredient to our algorithm is that we can replace the nested constraints~(\ref{sub_03}) by specific single-variable constraints without changing the optimal solution. By doing this, we transform an instance of QRAP-NC into an equivalent instance of QRAP. More precisely, we show that each subproblem QRAP-NC$^j(C)$ yields the same optimal solution as the following instance of QRAP:
\begin{align}
\text{QRAP}^j(C) \ : \ 
\min_{x \in \mathbb{R}^{j}}
&\sum_{i \in \mathcal{N}_j} \frac{1}{2} \frac{x_i^2}{a_i} \nonumber \\
\text{s.t. } & \sum_{i \in \mathcal{N}_j} x_i = C, \label{sub_alt_02} \\
& {x}_i^{j-1}(L_{j-1}) \leq x_i \leq {x}_i^{j-1}(U_{j-1}), \quad i \in \mathcal{N}_{j-1}, \label{sub_alt_03} \\
&l_j \leq x_j \leq u_j, \label{sub_alt_04}
\end{align} 
where the bounds ${x}^{j-1}(L_{j-1})$ and ${x}^{j-1}(U_{j-1})$ in (\ref{sub_alt_03}) are the optimal solutions of the problems QRAP-NC$^{j-1}(L_{j-1})$ and QRAP-NC$^{j-1}(U_{j-1})$ respectively. Note that the single-variable bounds for $x_j$ in (\ref{sub_alt_04}) are the same as those of the original subproblem QRAP-NC$^j(C)$.

The validity of this transformation is proven by Lemmas~\ref{lemma_mono}-\ref{lemma_add}. First, Lemma~\ref{lemma_mono} shows that the optimal solution ${x}^j(C)$ to the subproblem QRAP-NC$^j(C)$ is non-decreasing in $C$. Subsequently, Lemma~\ref{lemma_bound} uses this property to show that when adding the alternative single-variable bounds~(\ref{sub_alt_03}) to the problem formulation of QRAP-NC$^j(C)$, the optimal solution ${x}^j(C)$ to QRAP-NC$^j(C)$ is not cut off. Finally, Lemma~\ref{lemma_add} shows that the alternative single-variable bounds~(\ref{sub_alt_03}) are stronger than the nested constraints~(\ref{sub_03}).

\begin{lemma}
If $L_j \leq A \leq B \leq U_j$, we have ${x}^j(A) \leq {x}^j(B)$ for a given $j \in \mathcal{N}$.
\label{lemma_mono}
\end{lemma}
\begin{proof}
This proof is based on the proof of Theorem~2 in \cite{Vidal2018} and given in Appendix~\ref{pf_mono}.
\end{proof}

\begin{lemma}

For a given $j \in \mathcal{N}_{n-1}$ and $C \in [L_j , U_j]$, we have that ${x}^j_i (L_j) \leq {x}_i^{j+1}(C) \leq {x}_i^j (U_j)$.
\label{lemma_bound}
\end{lemma}
\begin{proof}
Let $x' := ({x}^{j+1}_1(C),\ldots, {x}^{j+1}_j(C))$ be the vector of the first~$j$ components of the optimal solution to the problem QRAP-NC$^{j+1}(C)$. Since $x'$ is feasible for all nested constraints~(\ref{sub_03}) for $k \in \mathcal{N}_j$, this vector is also the optimal solution to QRAP-NC$^j(A)$ where $A := \sum_{i \in \mathcal{N}_j} x'_i$, i.e., we have $x' = {x}^j(A)$. Since $A \in [L_{j} , U_{j}]$, Lemma~\ref{lemma_mono} implies that ${x}_i^j(L_j) \leq {x}_i^j(A) \leq {x}_i^j(U_j)$ for all $i \in \mathcal{N}_j$. It follows that ${x}_i^j(L_j) \leq {x}_i^{j+1}(C) \leq {x}_i^j(U_j)$ for all $i \in \mathcal{N}_j$.
\end{proof}

\begin{lemma}
If for a given $j \in \mathcal{N}_{n-1}$ and vector $y \in \mathbb{R}^{j}$ we have ${x}^j(L_j) \leq y \leq {x}^j(U_j)$, then $L_k \leq \sum_{i \in \mathcal{N}_k} y_i \leq U_k$ for all $k \in \mathcal{N}_j$. 
\label{lemma_add}
\end{lemma}
\begin{proof}
The sum of the inequalities ${x}_i^j(L_j) \leq y_i \leq {x}_i^j(U_j)$ over all $i \in \mathcal{N}_k$ yields
\begin{equation*}
\sum_{i \in \mathcal{N}_k} {x}_i^j(L_j) \leq \sum_{i \in \mathcal{N}_k} y_i \leq \sum_{i \in \mathcal{N}_k} {x}_i^j(U_j).
\end{equation*}
Since ${x}^j(L_j)$ and ${x}^j(U_j)$ are feasible for QRAP-NC$^j(L_j)$ and QRAP-NC$^j(U_j)$ respectively and $k \leq j$, we have $L_k \leq \sum_{i \in \mathcal{N}_k} {x}_i^j(L_j)$ and $\sum_{i \in \mathcal{N}_k} {x}_i^j(U_j) \leq U_k$ and the result of the lemma follows.
\end{proof}

Lemma~\ref{lemma_bound} implies that, given optimal solutions ${x}^{j-1}(L_{j-1})$ and ${x}^{j-1}(U_{j-1})$, we can replace the nested constraints~(\ref{sub_03}) in QRAP-NC$^j(C)$ by the single-variable bounds~(\ref{sub_alt_03}) without cutting off the optimal solution to QRAP-NC$^j(C)$. Moreover, since these single-variable bounds are stronger than the nested constraints by Lemma~\ref{lemma_add}, adding these constraints does not change the optimal objective value. It follows directly that any optimal solution to QRAP$^j(C)$ is also optimal for QRAP-NC$^j(C)$.

Based on Lemmas~\ref{lemma_mono}-\ref{lemma_add}, the following approach can be used to solve QRAP-NC. We successively solve the subproblems QRAP$^j(L_j)$ and QRAP$^j(U_j)$ from $j=1$ to $n-1$ and finally the subproblem QRAP$^n(R)$, whereby in each step we use the optimal solutions to the preceding subproblems QRAP$^{j-1}(L_{j-1})$ and QRAP$^{j-1}(U_{j-1})$ as input. Note that each of the subproblems is an instance of QRAP. This approach is summarized in Algorithm~\ref{alg_init}.

\begin{algorithm}[ht!]
\caption{An initial sequential algorithm for QRAP-NC.}
\label{alg_init}
\begin{algorithmic}[1]
\STATE{\textbf{Input:} Parameters $a \in \mathbb{R}^n_{>0}$, $L,U \in \mathbb{R}^{n-1}$, $R \in \mathbb{R}$, and $l,u \in \mathbb{R}^n$}
\STATE{\textbf{Output:} Optimal solution ${x}$ to QRAP-NC}
\STATE{Initialize ${x}_1^1(L_1) = L_1$; ${x}_1^1(U_1) = U_1$}
\FOR{$j=2,\ldots,n-1$}
\STATE{Compute optimal solutions ${x}^j(L_j)$ and ${x}^j(U_j)$ to QRAP$^j(L_j)$ and QRAP$^j(U_j)$ respectively}
\ENDFOR
\STATE{Compute optimal solution ${x}^n(R)$ to QRAP$^n(R)$}
\RETURN{Optimal solution $\bar{x} :={x}^n(R)$}
\end{algorithmic}
\end{algorithm}

Since each subproblem QRAP$^j(\cdot)$ can be solved in $O(n)$ time \cite{Brucker1984}, the worst-case time complexity of Algorithm~\ref{alg_init} is $O(n^2)$. However, linear-time algorithms for QRAP such as \cite{Brucker1984} attain their linear time complexity by employing linear-time algorithms for median finding, which are, as already mentioned, in general slower than simple sorting- or sampling-based approaches \cite{Kiwiel2005,Alexandrescu2017}. Note, that also the $O(n \log n)$ time algorithm in \cite{Vidal2018} attains its worst-case time complexity by using such slow linear-time algorithms as a subroutine.

In the next section, we propose an algorithm to solve QRAP-NC in $O(n \log n)$ time that, as opposed to the algorithm in \cite{Vidal2018}, does not require linear-time median-finding algorithms. Instead, it only requires a simple data structure for double-ended priority queues to store several bookkeeping parameters.

We conclude this section with two remarks that may be of independent interest:
\begin{enumerate}
\item
It can be shown that Lemmas~\ref{lemma_mono}-\ref{lemma_add} also hold for the case where the variables are integer-valued, i.e., $x \in \mathbb{Z}^n$ (see also Theorem~5 in \cite{Vidal2018}), given that all parameters $a$, $L$, $U$, $l$, and $u$ are also integer-valued and nonnegative. As a consequence, when solving each subproblem QRAP$^j(\cdot)$ with integer variables, Algorithm~\ref{alg_init} computes an optimal solution to QRAP-NC with integer variables. The worst-case time complexity of this algorithm is $O(n^2)$ since each QRAP$^j(\cdot)$ subproblem with integer variables can be solved in $O(j)$ time \cite{Ibaraki1988}.
\item
Lemmas~\ref{lemma_mono}-\ref{lemma_add} can be generalized to the case where the objective function is the sum of separable convex cost functions $f_i$, i.e., where we replace each term $\frac{1}{2} \frac{x_i^2}{a_i}$ by a convex function $f_i (x_i)$. For this more general problem, this leads to a sequential algorithm that is very similar to Algorithm~\ref{alg_init}. However, initial computational tests indicated that both this algorithm and Algorithm~\ref{alg_init} are in practice much slower than both \infe\ and \dec.
\end{enumerate}

\section{A fast $O(n \log n)$ time algorithm for QRAP-NC}
\label{sec_quad}
The sequential algorithm derived in the previous section does not match the best known time complexity of the algorithm in \cite{Vidal2018}. However, we show in this section that we can implement Algorithm~\ref{alg_init} such that its time complexity reduces to $O( n \log n)$ without requiring a linear-time median finding algorithm. Instead, we only require a data type that supports insertion of elements and the finding and removing of minimum and maximum elements in $O(\log n)$ time such as a double-ended priority queue.

The key to efficiency in our approach is that we do not explicitly compute the solution to each QRAP subproblem. Instead, we only compute an optimal Lagrange multiplier corresponding to the resource constraint~(\ref{sub_alt_02}) that characterizes the entire optimal solution to this subproblem. Subsequently, we use these multipliers to reconstruct the optimal solution to the original problem QRAP-NC using two sets of simple recursive relations that can be executed in $O(n)$ time. In order to compute the Lagrange multipliers without explicitly storing intermediate solutions, we exploit the special structure of these multipliers and of a specific algorithm for solving QRAP. 

First, in Section~\ref{sec_notation}, we introduce some of the used notation. Second, in Section~\ref{sec_opt_mult}, we derive an efficient approach for computing the optimal Lagrange multipliers of the subproblems QRAP$^j(L_j)$ and QRAP$^j(U_j)$. Based on these optimal Lagrange multipliers, we derive in Section~\ref{sec_retrieve}, two simple recursions to compute the optimal solution ${x}$ to QRAP-NC. Finally, in Section~\ref{sec_NC_alg}, we present an $O(n \log n)$ algorithm for QRAP-NC and discuss an implementation that attains this worst-case time complexity.

\subsection{Notation}
\label{sec_notation}
We introduce the following notation concerning the subproblems QRAP$^j(L_j)$ and QRAP$^j(U_j)$ that we use throughout the remainder of this article. We denote for $j \in \mathcal{N}$ the lower and upper single variable bounds (\ref{sub_alt_03}) and (\ref{sub_alt_04}) of QRAP$^j(C)$ with $C \in [L_j , U_j]$ by $\bar{l}^j := (\bar{l}_i^j)_{i \in \mathcal{N}_j}$ and $\bar{u}^j := (\bar{u}_i^j)_{i \in \mathcal{N}_j}$, where $\bar{l}_i^j := {x}_i^{j-1}(L_{j-1})$ and $\bar{u}_i^j := {x}_i^{j-1}(U_{j-1})$ for $i < j$, and $\bar{l}_j^j := l_j$ and $\bar{u}_j^j := u_j$. Furthermore, we denote by $\alpha^j := (\alpha_i^j)_{i \in \mathcal{N}_j}$ and $\beta^j := (\beta_i^j)_{i \in \mathcal{N}_j}$ the lower and upper breakpoints for the QRAP$^j(C)$ subproblem. We call the breakpoints corresponding to $i=j$, i.e., $\alpha_j^j$ and $\beta_j^j$, \emph{initial} breakpoints since QRAP$^j(C)$ is the first subproblem, i.e., with lowest index~$j$, in which we have to compute breakpoint values for the variable $x_j$. Note that we can compute these breakpoints directly as $\alpha_j^j := \frac{l_j}{a_j}$ and $\beta_j^j := \frac{u_j}{a_j}$ by definition of the subproblem QRAP$^j(C)$. 

Furthermore, let $\kappa^{j}$ and $\lambda^{j}$ denote the optimal Lagrange multipliers for the subproblems QRAP$^{j}(L_{j})$ and QRAP$^{j}(U_{j})$ respectively and define $\kappa := (\kappa^j)_{j \in \mathcal{N}}$ and $\lambda := (\lambda^j)_{j \in \mathcal{N}}$, where we set $\kappa^1 := \alpha_1^1$ and $\lambda^1 := \beta_1^1$. If the optimal Lagrange multiplier for a given subproblem QRAP$^j(L_j)$ is not unique, we define without loss of generality $\kappa^j$ as the \emph{maximum} optimal Lagrange multiplier. Analogously, we define $\lambda^j$ as the \emph{minimum} optimal Lagrange multiplier of subproblem QRAP$^j(U_j)$. Note that ${x}^j(L_j) = {x}^j[\kappa^j]$ and ${x}^j(U_j) = {x}^j[\lambda^j]$ by definition of the subproblems QRAP$^j(L_j)$ and QRAP$^j(U_j)$ and of $\kappa^j$ and $\lambda^j$. Finally, for a given subproblem QRAP$^j(C)$, we define the set of its lower breakpoints as $\mathcal{A}^j := \lbrace \alpha_i^j \ | \ i \in \mathcal{N}_j \rbrace$ and the set of its upper breakpoints as $\mathcal{B}^j := \lbrace \beta_i^j \ | \ i \in \mathcal{N}_j \rbrace$. Recall that in Section~\ref{sec_BP_search} we defined breakpoint sets as \emph{multi}sets for convenience when solving QRAP. However, for our approach for a fast algorithm for QRAP-NC, it is crucial that the breakpoint sets do not contain duplicate elements. Therefore, in this section and the remainder of this article, we regard $\mathcal{A}^j$ and $\mathcal{B}^j$ as ordinary sets.

\subsection{Computing the optimal Lagrange multipliers of the subproblems}
\label{sec_opt_mult}

The goal of this subsection is to derive an efficient approach for computing the optimal Lagrange multiplier of each QRAP subproblem in Algorithm~\ref{alg_init} without explicitly calculating any of the intermediate optimal solutions ${x}^j(L_j)$ and ${x}^j(U_j)$ for $j \in \mathcal{N}$.  If we would follow the latter strategy, i.e., if we solve each pair of subproblems QRAP$^j(L_j)$ and QRAP$^j(U_j)$ from scratch, e.g., using Algorithm~\ref{alg_QRAP_01}, we would have to explicitly compute the breakpoint sets for each pair of subproblems. This leads to $O(n^2)$ computations and thus forms an efficiency bottleneck within this algorithm.

We show that we can apply the breakpoint search procedure in Algorithm~\ref{alg_QRAP_01} for solving the subproblems such that each breakpoint set $\mathcal{A}^{j+1}$ can be obtained from the previous set $\mathcal{A}^j$ in $O(1)$ amortized steps, i.e., the total number of steps required to carry out this construction for all $j \in \mathcal{N}_{n-1}$ is $O(n)$. This can be done because of two intermediate results that we show in this subsection. First, the number of distinct values that the breakpoints can take is not $O(n^2)$ but $O(n)$. We obtain this result by unveiling a useful relation between breakpoints of consecutive subproblems, i.e., between $\alpha^j$, $\beta^j$ and $\alpha^{j+1}$, $\beta^{j+1}$. Second, when constructing the breakpoint sets, each distinct breakpoint value is included in or removed from a breakpoint set at most twice during the entire procedure. For this, it is important that we solve each lower subproblem QRAP$^j(L_j)$ by considering the breakpoints in non-decreasing order and each upper subproblem QRAP$^j(U_j)$ by considering the breakpoints in non-increasing order. Together, these two results imply that the construction of the breakpoint sets requires in total $O(n)$ additions and removals of breakpoint values. By using an appropriate data structure such as double-ended priority queues for maintaining the breakpoint sets, each of these steps can be executed in $O(\log n)$ time, which leads to an overall $O(n \log n)$ complexity for computing the optimal Lagrange multipliers $\kappa$ and $\lambda$.

The outline of the remainder of this subsection is as follows. First, in Section~\ref{sec_rel_BP}, we analyze the relation between breakpoints of consecutive subproblems and show that the number of distinct breakpoint values is $O(n)$. Subsequently, in Section~\ref{sec_construct}, we use this information and the structure of Algorithm~\ref{alg_QRAP_01} to construct the breakpoint sets for each subproblem from those of the predecing subproblems. Finally, in Section~\ref{sec_update}, we discuss how the updating of the bookkeeping parameters within the breakpoint search procedure must be adjusted when applying this procedure to the subproblems QRAP$^j(L_j)$ and QRAP$^j(U_j)$.

\subsubsection{Relation between consecutive breakpoints}
\label{sec_rel_BP}

We first show how we can efficiently obtain the breakpoint set of a given subproblem QRAP$^{j+1}(C)$ based on the breakpoint set and optimal Lagrange multipliers of the preceding subproblems QRAP$^{j}(L_j)$ and QRAP$^{j}(U_j)$. We establish for a given $j \in \mathcal{N}_{n-1}$ and $i<j$ the following relation between the subsequent lower breakpoints $\alpha_i^j$ and $\alpha_i^{j+1}$:
\begin{itemize}
\item If $\kappa^j < \alpha_i^j$, it follows from Equation~(\ref{eq_opt_QRA}) that ${x}_i^{j}[\kappa^j] = \bar{l}_i^j$ since $\alpha_i^j = \frac{\bar{l}_i^j}{a_i}$. This implies that $\bar{l}_i^{j+1} = {x}_i^j(L_j) = {x}_i^{j}[\kappa^j]= \bar{l}_i^j$ and thus $\alpha_i^{j+1}= \alpha_i^j$.
\item If $\alpha_i^j \leq \kappa^j < \beta_i^j$, it follows from Equation~(\ref{eq_opt_QRA}) that ${x}_i^j[\kappa^j] = a_i \kappa^j$. Thus, $\alpha_i^{j+1} = \frac{\bar{l}_i^{j+1}}{a_i} = \frac{{x}_i^j(L_j)}{a_i} = \frac{{x}_i^j[\kappa^j]}{a_i}= \kappa^j$.
\item If $\beta_i^j \leq \kappa^j$, then it follows from Equation~(\ref{eq_opt_QRA}) that ${x}_i^j[\kappa^j] = \bar{u}_i^{j}$. This implies that $\bar{l}_i^{j+1} = {x}_i^j(L_j) = {x}_i^j[\kappa^j] = \bar{u}_i^j$ and thus $\alpha_i^{j+1} = \beta_i^j$.
\end{itemize}
Summarizing, we can determine $\alpha_i^{j+1}$ from the previous breakpoints $\alpha_i^j$ and $\beta_i^j$ and the optimal Lagrange multiplier $\kappa^j$ as follows:
\begin{equation}
\alpha_i^{j+1} = \begin{cases}
\alpha_i^j & \text{if } \kappa^j < \alpha_i^j, \\
\kappa^j & \text{if } \alpha_i^j \leq \kappa^j < \beta_i^j, \\
\beta_i^j & \text{if } \beta_i^j \leq \kappa^j.
\end{cases}
\label{eq_rel_lower_BP}
\end{equation}
Analogously, we obtain the following expression for the upper breakpoint $\beta_i^{j+1}$ in terms of the previous breakpoints $\alpha_i^j$ and $\beta_i^j$ and the optimal Lagrange multiplier $\lambda^j$:
\begin{equation}
\beta_i^{j+1} = \begin{cases}
\beta_i^j & \text{if } \lambda^j > \beta_i^j, \\
\lambda^j & \text{if } \beta_i^j \geq \lambda^j > \alpha_i^j, \\
\alpha_i^j & \text{if } \alpha_i^j \geq \lambda^j.
\end{cases}
\label{eq_rel_upper_BP}
\end{equation}
Note that it follows from these relations that $\alpha_i^j \leq \alpha_i^{j+1}$ and $\beta_i^j \geq \beta_i^{j+1}$ for each $j \in \mathcal{N}_{n-1}$. Moreover, note that the only values that the breakpoints can take are those of the initial breakpoints $\alpha_j^j$ and $\beta_j^j$ or of the optimal Lagrange multipliers in $\kappa$ and $\lambda$. Thus, the number of distinct values  among all breakpoints is limited by $4n$.

\subsubsection{Constructing consecutive breakpoint sets}
\label{sec_construct}

As observed at the end of Section~\ref{sec_BP_search}, we can solve a given QRAP subproblem by searching its breakpoints either in non-decreasing or non-increasing order. In particular, we can solve all lower subproblems QRAP$^j(L_j)$ by searching the breakpoints in non-decreasing order and all upper subproblems QRAP$^j(U_j)$ by searching the breakpoints in non-increasing order. When doing this, note that for solving the upper subproblem QRAP$^j(U_j)$ we can use as breakpoint sets the sets that ``remain'' from the breakpoint search for the lower subproblem. More precisely, instead of the sets $\mathcal{A}^j$ and $\mathcal{B}^j$ that we also use as breakpoint sets for the lower subproblem QRAP$^j(L_j)$, we can use the sets $\lbrace  \alpha_i^j \in \mathcal{A}^j \ | \ \alpha_i^j \geq \kappa^j \rbrace$ and $\lbrace  \beta_i^j \in \mathcal{B}^j \ | \  \beta_i^j  \geq \kappa^j \rbrace$ respectively. This is because $\kappa^j \leq \lambda^j$ and thus in the breakpoint search for the upper problem QRAP$^j(U_j)$ no breakpoints smaller than $\kappa^j$ need to be considered.

We define the sets $\tilde{\mathcal{A}}^j$ and $\tilde{\mathcal{B}}^j$ as the sets of lower and upper breakpoints that remain to be considered after solving the subproblems QRAP$^j(L_j)$ and QRAP$^j(U_j)$ in the way described in the previous paragraph, i.e., we have
\begin{align*}
\tilde{\mathcal{A}}^j &:= \lbrace  \alpha_i^j \in \mathcal{A}^j \ | \ \kappa^j \leq \alpha_i^j \leq \lambda^j \rbrace, \\
\tilde{\mathcal{B}}^j &:= \lbrace  \beta_i^j \in \mathcal{B}^j \ | \ \kappa^j \leq \beta_i^j \leq \lambda^j \rbrace.
\end{align*}
We call these sets the \emph{remaining breakpoint sets} of the subproblems QRAP$^j(L_j)$ and QRAP$^j(U_j)$. In the following, we relate these two remaining breakpoint sets to the breakpoint sets of the next two subproblems, i.e., to the sets $\mathcal{A}^{j+1}$ and $\mathcal{B}^{j+1}$. For this, we focus on the relation between the lower remaining breakpoint sets $\tilde{\mathcal{A}}^j$ and the lower breakpoint set $\mathcal{A}^{j+1}$; the relation between the upper remaining breakpoint set $\tilde{\mathcal{B}}^j$ and the upper breakpoint set $\mathcal{B}^{j+1}$ is analogous.

For each $i \in \mathcal{N}_j$, we consider four cases for the value of $\alpha_i^{j+1}$:

\begin{enumerate}
\item
If $\kappa^j \leq \alpha_i^j \leq \lambda^j$, it follows from Equation~(\ref{eq_rel_lower_BP}) that $\alpha_i^{j+1} = \alpha_i^j$. Thus, all values in $\tilde{\mathcal{A}}^j$ act as breakpoint values for the next subproblems, i.e., $\tilde{\mathcal{A}}^j \subseteq \mathcal{A}^{j+1}$.
\item
If $\alpha_i^j < \kappa^j$ and $\kappa^j < \beta_i^j$, it follows from Equation~(\ref{eq_rel_lower_BP}) that $\alpha_i^{j+1} = \kappa^j$.
\item
If $\alpha_i^j < \kappa^j$ and $\beta_i^j \leq \kappa^j$, it follows from Equations~(\ref{eq_rel_lower_BP}) and~(\ref{eq_rel_upper_BP}) that ${x}_i^j(L_j) = \bar{u}_i^j$ and $\beta_i^{j+1} = \beta_i^{j} = \alpha_i^{j+1}$ respectively. Thus, 
\begin{equation*}
\beta_i^{j+1} = \alpha_i^{j+1} \leq \dots \leq \alpha_i^n \leq \beta_i^n \leq \dots \leq \beta_i^{j+1}.
\end{equation*}
This means that $\alpha_i^{j'} = \beta_i^{j'} = \beta_i^{j}$ and $\bar{l}_i^{j'} = \bar{u}_i^{j'} = \bar{u}_i^j$ for all $j' > j$. Thus, in all remaining subproblems, the lower and upper breakpoints of $i$ coincide and ${x}_i^{j'}(C) = \bar{u}_i^j$ for any $j' > j$ and $L_{j'} \leq C \leq U_{j'}$, regardless of the values of the future optimal Lagrange multipliers $\kappa^{j'}$ and $\lambda^{j'}$. This means that we can remove this index (variable) from the breakpoint search.
\item
Finally, if $\alpha_i^j > \lambda^j$, it follows from Equations~(\ref{eq_rel_lower_BP}) and~(\ref{eq_rel_upper_BP}) that $\alpha_i^{j+1} = \alpha_i^j$ and $\beta_i^{j+1} = \alpha_i^j$ respectively. Thus, $\alpha_i^{j+1} = \beta_i^{j+1} = \alpha_i^j$. Analogously to the case $\alpha_i^j \leq \beta_i^j < \kappa^j$, it follows that $\bar{l}_i^{j'} = \bar{u}_i^{j'} = \bar{l}_i^j$ and ${x}_i^{j'}(C) = \bar{l}_i^j$ for all $j' > j$ and $L_{j'} \leq C \leq U_{j'}$. Thus, also in this case we can remove the index~$i$ from the breakpoint search.
\end{enumerate}
These four cases imply that we can construct $\mathcal{A}^{j+1}$ from $\tilde{\mathcal{A}}^j$ as follows:
\begin{equation*}
\mathcal{A}^{j+1} = \tilde{\mathcal{A}}^j \cup \lbrace \alpha_{j+1}^{j+1} \rbrace \cup
\begin{cases}
\lbrace \kappa^j \rbrace & \text{if there exists } i \text{ such that } \alpha_i^j < \kappa^j < \beta_i^j, \\
\emptyset & \text{otherwise.}
\end{cases}
\end{equation*}
Analogously, we can construct $\mathcal{B}^{j+1}$ from $\tilde{\mathcal{B}}^j$ as follows:
\begin{equation*}
\mathcal{B}^{j+1} = \tilde{\mathcal{B}}^j \cup \lbrace \beta_{j+1}^{j+1} \rbrace \cup
\begin{cases}
\lbrace \lambda^j \rbrace & \text{if there exists } i \text{ such that } \alpha_i^j < \lambda^j < \beta_i^j, \\
\emptyset & \text{otherwise.}
\end{cases}
\end{equation*}
The above constructions show how the breakpoint sets evolve over the course of the algorithm. First, in this construction, at most $4n$ additions of breakpoint values to a breakpoint set occur. Second, during the breakpoint search procedure of Algorithm~\ref{alg_QRAP_01}, breakpoints are only removed and not added. This means that updating the breakpoint steps can be done in $O(n)$ steps, i.e., by $O(n)$ additions and removals of breakpoint values.

\subsubsection{Updating bookkeeping parameters}
\label{sec_update}

In order to to efficiently compute the sums ${z}^j[\delta] := \sum_{i \in \mathcal{N}_j} {x}_i^j[\delta]$ for a given breakpoint $\delta$, we define the following bookkeeping parameters analogously to those in the breakpoint search procedure for QRAP in Algorithm~\ref{alg_QRAP_01}:
\begin{align*}
P^j(\delta) &:=  \sum_{i \leq j: \ \delta < \frac{\bar{l}_i^j}{a_i}} \bar{l}_i^j
+
\sum_{i \leq j: \ \delta^j \geq \frac{\bar{u}_i^j}{a_i}} \bar{u}_i^j; &
Q^j (\delta) &:= \sum_{i \leq j: \ \frac{\bar{l}_i^j}{a_i} \leq \delta < \frac{\bar{u}_i^j}{a_i}} a_i ;
\end{align*}
Each breakpoint value $\kappa^{j'}$ and $\lambda^{j'}$ in a given breakpoint set acts as a collective breakpoint for one or multiple activities. As a consequence, within the breakpoint search procedure, they have the same function as the ``regular'' initial lower and upper breakpoint values $\alpha_i^i$ and $\beta_i^i$. Thus, when a breakpoint value of the form $\kappa^{j'}$ or $\lambda^{j'}$ has been considered, we require an efficient update of the bookkeeping sums $P^j(\kappa^{j'})$, $Q^j(\kappa^{j'})$ or $P^j(\lambda^{j'})$, $Q^j(\lambda^{j'})$ respectively. In the case of $\kappa^{j'}$, we update $P^j(\kappa^{j'})$ by subtracting from this value the sum of the lower bounds $\bar{l}_i^j$ of those activities $i$ whose lower breakpoint equals $\kappa^{j'}$, i.e., for which $\alpha_i^j = \kappa^{j'}$. The sum of these values is
\begin{equation*}
\sum_{i < j: \ \alpha_i^j = \kappa^{j'}} \bar{l}_i^j
=
\sum_{i < j: \ \alpha_i^j = \kappa^{j'}} a_i \alpha_i^j
=
\sum_{i < j: \ \alpha_i^j = \kappa^{j'}} a_i \kappa^{j'}
=
\kappa^{j'} \sum_{i \leq j': \ \alpha_i^{j'} \leq \kappa^{j'} < \beta_i^{j'}} a_i 
=
Q^{j'} (\kappa^{j'}) \kappa^{j'},
\end{equation*}
since $\alpha_i^j = \frac{\bar{l}_i^j}{a_i}$ for each $i \in \mathcal{N}_j$ and we have that $\alpha_i^j = \kappa^{j'}$ if and only if $\alpha_i^{j''} = \kappa^{j'}$ for all $j'' \in \lbrace j', \ldots, j \rbrace$. Analogously, we update the bookkeeping sum $Q^j(\kappa^{j'})$ by adding to this value the sum of the parameters $a_i$ for those $i$ with $\alpha_i^j = \kappa^{j'}$. This sum is
\begin{equation*}
\sum_{i < j: \ \alpha_i^j = \kappa^{j'}} a_i
=
 \sum_{i \leq j': \ \alpha_i^{j'} \leq \kappa^{j'} < \beta_i^{j'}} a_i 
=
Q^{j'} (\kappa^{j'}).
\end{equation*}
Thus, the updates take the form $P^j(\kappa^{j'}) - Q^{j'}(\kappa^{j'}) \kappa^{j'}$ and $Q^j(\kappa^{j'}) + Q^{j'}(\kappa^{j'})$.

The updates for the case of $\lambda^{j'}$, i.e., for $P^j(\lambda^{j'})$ and $Q^j(\lambda^{j'})$, are analogous to those for the case of $\kappa^{j'}$. Table~\ref{tab_rel_02} provides an overview of the updates of the bookkeeping sums for both these cases for each of the four breakpoint values types $\alpha_i^i$, $\beta_i^i$, $\kappa^{j'}$, and $\lambda^{j'}$.

\begin{table}[ht!]
\centering
\begin{tabular}{l | ll | ll}
\toprule
& \multicolumn{2}{c|}{In QRAP$^j(L_j)$ (non-decreasing search)} & \multicolumn{2}{c}{In QRAP$^j(U_j)$ (non-increasing search)} \\
Type of $\delta$ & $P^j(\delta)$ & $Q^j(\delta)$ & $P^j(\delta)$ & $Q^j(\delta)$ \\
\midrule
$\delta \equiv \alpha_i^i$ & $P^j(\delta) - l_i$ & $Q^j(\delta)_i + a_i$ & $P^j(\delta) + l_i$ & $Q^j(\delta) - a_i$ \\
$\delta \equiv \beta_i^i$ & $P^j(\delta) + u_i$ & $Q^j(\delta) - a_i$ & $P^j(\delta) - u_i$ & $Q^j(\delta) + a_i$ \\
$\delta \equiv \kappa^{j'}$, $j' < j$ & $P^j(\delta)  - Q^{j'}(\kappa^{j'}) \kappa^{j'}$ & $Q^j(\delta) + Q^{j'}(\kappa^{j'})$ & $P^j(\delta)  + Q^{j'}(\kappa^{j'}) \kappa^{j'}$ & $Q^j(\delta) - Q^{j'}(\kappa^{j'})$ \\
$\delta \equiv \lambda^{j'}$, $j' < j$ & $P^j(\delta)  + Q^{j'}(\lambda^{j'}) \lambda^{j'}$ & $Q^j(\delta) - Q^{j'}(\lambda^{j'})$ & $P^j(\delta)  - Q^{j'}(\lambda^{j'}) \lambda^{j'}$ & $Q^j(\delta) + Q^{j'}(\lambda^{j'})$ \\
\bottomrule
\end{tabular}
\caption{Updating the bookkeeping sums $P^j(\delta)$ and $Q^j(\delta)$ when searching the breakpoints in non-decreasing order (QRAP$^j(L_j)$) and non-increasing order (QRAP$^J(U_j)$).}
\label{tab_rel_02}
\end{table}

\subsection{Recovering the optimal solution to QRAP-NC}
\label{sec_retrieve}

In the previous section, we found an efficient way to compute the optimal Lagrange multipliers $\kappa^j$ and $\lambda^j$ for the QRAP subproblems QRAP$^j(L_j)$ and QRAP$^j(U_j)$. In this section, we show how we can use these values to compute the optimal solution ${x}^n(R)$. For this, we first determine which nested constraints are tight in ${x}^n(R)$ and use this information to reconstruct the individual terms ${x}_i^n (R)$ for $i \in \mathcal{N}$. To this end, for each $j \in \mathcal{N}_{n-1}$, let ${\ell}_j$ denote the smallest index larger than or equal to $j$ such that one of its corresponding nested constraints is tight in ${x}^n(R)$. More precisely,
\begin{equation*}
{\ell}_j  := \min \left( k \geq j \ \middle| \ \sum_{i \in \mathcal{N}_k} {x}_i^n (R) = L_k \text{ or } \sum_{i \in \mathcal{N}_k} {x}_i^n (R) = U_k \right).
\end{equation*}
Furthermore, let $V_j$ denote the value of the tight nested constraint corresponding to the index ${\ell}_j$ and $\chi^j$ the corresponding multiplier, i.e., $V_j \in \lbrace L_{{\ell}_j}, U_{{\ell}_j} \rbrace$ and $\chi^j \in \lbrace  \kappa^{{\ell}_j} , \lambda^{{\ell}_j} \rbrace$. More precisely,
\begin{itemize}
\item
$\sum_{i \in \mathcal{N}_{{\ell}_j}} {x}_i^n (R) = L_{{\ell}_j}$ implies $V_j = L_{{\ell}_j}$ and $\chi^j = \kappa^{{\ell}_j}$;
\item
$\sum_{i \in \mathcal{N}_{{\ell}_j}}  {x}_i^n (R) = U_{{\ell}_j}$ implies $V_j = U_{{\ell}_j}$ and $\chi^j = \lambda^{{\ell}_j}$.
\end{itemize}
The main result in this subsection is that the values $\chi^j$ act as optimal Lagrange multipliers for the resource constraint~(\ref{sub_alt_02}) in the subproblem QRAP$^n(R)$. As a consequence, given these values, we can calculate ${x}(R)$ directly using a relation similar to the Lagrangian relaxation solution in Equation~(\ref{eq_opt_QRA}). To show this result, we prove Lemmas~\ref{lemma_chi} and~\ref{lemma_xn}. First, Lemma~\ref{lemma_chi} shows how we can iteratively compute $\chi$ from the optimal multipliers $\kappa$ and $\lambda$ using a simple recursive relation. Second, Lemma~\ref{lemma_xn} shows how we can calculate ${x}^n (R)$ from $\chi$ using a relation similar to that in Equation~(\ref{eq_opt_QRA}).
\begin{lemma}
We have $\chi^n = \kappa^n = \lambda^n$. Moreover, for each $j \in \mathcal{N}_{n-1}$, we have:
\begin{enumerate}
\item $ \chi^{j+1} \leq \kappa^j$ implies $\sum_{i \in \mathcal{N}_j} {x}_i^n (R) = L_j$ and $\chi^j = \kappa^j$;
\item $\lambda^j \leq \chi^{j+1}$ implies $\sum_{i \in \mathcal{N}_j} {x}_i^n (R) = U_j$ and $\chi^j = \lambda^j$,
\item $\kappa^j < \chi^{j+1} < \lambda^{j}$ implies $L_j < \sum_{i \in \mathcal{N}_j} {x}_i^n (R) < U_j$ and $\chi^j = \chi^{j+1}$.
\end{enumerate}
\label{lemma_chi}
\end{lemma}
\begin{proof}
See Appendix~\ref{pf_chi}.
\end{proof}

\begin{lemma}
For each $i \in \mathcal{N}$, we have
\begin{equation}
{x}_i^n (R) =
\begin{cases}
l_i & \text{if } \chi^i < \alpha_i^i, \\
a_i \chi^i & \text{if } \alpha_i^i \leq \chi^i < \beta_i^i, \\
u_i & \text{if } \beta_i^i \leq \chi^i.
\end{cases}
 \label{eq_opt_nL}
\end{equation}
\label{lemma_xn}
\end{lemma}
\begin{proof}
See Appendix~\ref{pf_xn}.
\end{proof}
Note that, starting from $\chi^n = \kappa^n$ and using Lemma~\ref{lemma_chi}, we can compute the values $\chi^{j}$ recursively as
\begin{equation}
\chi^j = \begin{cases}
\kappa^j & \text{if } \chi^{j+1} \leq \kappa^j, \\
\lambda^j & \text{if } \chi^{j+1} \geq \lambda^j, \\
\chi^{j+1} & \text{otherwise.}
\end{cases}
\label{eq_chi}
\end{equation}
Thus, given the optimal Lagrange multipliers $\kappa$ and $\lambda$, we can compute ${x}$ in $O(n)$ time as ${x}^n(R)$ using the two relatively simple recursions in Equations~(\ref{eq_opt_nL}) and~(\ref{eq_chi}).

\subsection{An $O(n \log n)$ time algorithm for QRAP-NC}
\label{sec_NC_alg}

In the previous two subsections, we derived an efficient approach to compute the optimal Lagrange multipliers $\kappa$ and $\lambda$ for the QRAP$^j(L_j)$ and QRAP$^j(U_j)$ subproblems and to compute from these multipliers the optimal solution ${x}$. In this subsection, we combine these two ingredients to formulate a fast and efficient algorithm for QRAP-NC (Algorithm~\ref{alg_01}). More precisely, in the first part of this subsection, Section~\ref{sec_implement}, we present our algorithm and discuss several of its details regarding the subroutines for computing the optimal Lagrange multipliers of the QRAP$^j(L_j)$ and QRAP$^j(U_j)$ subproblems. This includes several procedures that deal with corner cases and with the updating of the breakpoint sets and the bookkeeping parameters. In the second part, Section~\ref{sec_complexity}, we focus on the efficiency of the algorithm. In particular, we prove in Lemma~\ref{lemma_time} that the algorithm has an $O(n \log n)$ worst-case time complexity when using an appropriate data structure.

\subsubsection{Description of the algorithm}
\label{sec_implement}

Algorithm~\ref{alg_01} captures our approach for solving QRAP-NC. First, in Lines~3-13, the algorithm initializes all problem parameters, the initial breakpoint values and breakpoint sets, and the initial bookkeeping parameters. Throughout the entire algorithm, it maintains four separate sets $\mathcal{A}$, $\mathcal{B}$, $\mathcal{K}$, and $\mathcal{L}$ of breakpoint values corresponding to the ``source'' of the values, i.e., this specifies whether they are one of the initial breakpoint values $\alpha_i^i$ or $\beta_i^i$ or one of the optimal Lagrange multipliers $\kappa^j$ or $\lambda^j$ respectively. Second, in Lines~14-16, the algorithm applies for each $j \in \mathcal{N} \backslash \lbrace 1 \rbrace$ the procedure {\sc SolveSubproblems}($j$) (see Algorithm~\ref{alg_BP}) that computes the optimal Lagrange multipliers $\kappa^j$ and $\lambda^j$ for the two subproblems QRAP$^j(L_j)$ and QRAP$^j(U_j)$. Finally, using the obtained vectors of optimal Lagrange multipliers $\kappa$ and $\lambda$, the algorithm computes in Lines~17-22 the (alternative) multiplier values $\chi$ using the recursion in Equation~(\ref{eq_chi}) and from these values the solution ${x}^n(R)$ using Equation~(\ref{eq_opt_nL}). 

\begin{algorithm}[ht!]
\caption{An $O(n \log n)$ time algorithm for QRAP-NC.}
\label{alg_01}
\begin{multicols}{2}
\begin{algorithmic}[1]
\STATE{\textbf{Input:} Parameters $a \in \mathbb{R}^n_{>0}$, $L,U \in \mathbb{R}^{n-1}$, $R \in \mathbb{R}$, and $l,u \in \mathbb{R}^n$}
\STATE{\textbf{Output:} Optimal solution ${x}$ to QRAP-NC}
\STATE{$L_1 = \max(L_1,l_1)$; $U_1 = \min (U_1,u_1)$}
\FOR{$j=2$ to $n$}
\STATE{$L_j = \max(L_j,L_{j-1} + l_j)$}
\STATE{$U_j = \min(U_j , U_{j-1} + u_j)$}
\ENDFOR
\FOR{$i=1$ to $n$}
\STATE{$\alpha_i^i = \frac{l_i}{a_i}$; $\beta_i^i = \frac{u_i}{a_i}$}
\ENDFOR
\STATE{$\kappa^1 = \alpha_1^1$; $\lambda^1 = \beta_1^1$; $\kappa^j = \infty$, $\lambda^j = -\infty$ for $j > 1$}
\STATE{Initialize breakpoint sets: $\mathcal{A} := \lbrace \alpha_1^1 \rbrace$; $\mathcal{B} := \lbrace \beta_1^1 \rbrace$; $\mathcal{K} := \emptyset$; $\mathcal{L} := \emptyset$}
\STATE{Initialize bookkeeping sums: $\bar{P}^1_L = \bar{P}^1_U = 0$; $\bar{Q}^1_L = \bar{Q}^1_U = a_1$}
\FOR{$j=2$ to $n$}
\STATE{Apply procedure {\sc SolveSubproblems}($j$)}
\ENDFOR
\STATE{$\chi^n := \kappa^n$}
\FOR{$i = n-1$ down to $1$}
\STATE{Compute $\chi^i$ using Equation~(\ref{eq_chi})}
\ENDFOR
\STATE{Compute ${x}^n(R)$ using Equation~(\ref{eq_opt_nL})}
\RETURN{Optimal solution $\bar{x} := {x}^n(R)$}
\end{algorithmic}
\end{multicols}
\end{algorithm}

\begin{algorithm}[ht!]
\caption{Procedure {\sc SolveSubproblems}($j$).}
\label{alg_BP}
\begin{multicols}{2}
\begin{algorithmic}[1]
\IF{$L_{j-1} + \max(l_j,\min(a_j \kappa^{j-1}, u_j)) = L_j$}
\STATE{$\kappa^j := \kappa^{j-1}$; replace $\kappa^{j-1}$ in $\mathcal{K}$ by $\kappa^j$}
\IF{$\kappa^j < \alpha_j^j$}
\STATE{$\bar{P}^j_L := \bar{P}^{j-1}_L + l_j$; $\bar{Q}^j_L := \bar{Q}^{j-1}_L$}
\ELSIF{$\beta_j^j < \kappa^j$}
\STATE{$\bar{P}^j_L = \bar{P}^{j-1}_L + u_j$; $\bar{Q}^j_L := \bar{Q}^{j-1}_L$}
\ELSE
\STATE{$\bar{P}^j_L := \bar{P}^{j-1}_L$; $\bar{Q}^j_L := \bar{Q}^{j-1}_L + a_j$}
\ENDIF
\ELSIF{$L_{j-1} + \max(l_j,\min(a_j \kappa^{j-1},u_j)) > L_j$}
\STATE{($\kappa^j < \kappa^{j-1}$:) $\kappa^j = (L_j - L_{j-1})/a_j$}
\STATE{$\bar{P}^j_L := L_{j-1}$; $\bar{Q}^j_L := a_j$}
\STATE{Add $\kappa^j$ to $\mathcal{K}$}
\ELSE
\STATE{($\kappa^j > \kappa^{j-1}$:) remove $\kappa^{j-1}$ from $\mathcal{K}$}
\ENDIF

\IF{$U_{j-1} + \max(l_j,\min(a_i \lambda^{j-1},u_j)) = U_j$}
\STATE{$\lambda^j := \lambda^{j-1}$; replace $\lambda^{j-1}$ in $\mathcal{L}$ by $\lambda^j$}
\IF{$\lambda^j < \alpha_j^j$}
\STATE{$\bar{P}^j_U := \bar{P}^{j-1}_U + l_j$, $\bar{Q}^j_U := \bar{Q}^{j-1}_U$}
\ELSIF{$\beta_j^j < \lambda^j$}
\STATE{$\bar{P}^j_U = \bar{P}^{j-1}_U + u_j$; $\bar{Q}^j_U := \bar{Q}^{j-1}_U$}
\ELSE
\STATE{$\bar{P}^j_U := \bar{P}^{j-1}_U$; $\bar{Q}^j_U := \bar{Q}^{j-1}_U + a_j$}
\ENDIF
\ELSIF{$U_{j-1} + \max(l_j,\min(a_j \lambda^{j-1},u_j)) < U_j$}
\STATE{($\lambda^j > \lambda^{j-1}$:) $\lambda^j = (U_j - U_{j-1})/a_j$}
\STATE{$\bar{P}^j_U := U_{j-1}$; $\bar{Q}^j_U := a_j$}
\STATE{Add $\lambda^j$ to $\mathcal{L}$}
\ELSE
\STATE{($\lambda^j < \lambda^{j-1}$:) remove $\lambda^{j-1}$ from $\mathcal{L}$}
\ENDIF

\IF{$\min(\kappa^{j-1},\kappa^{j}) < \alpha_j^j \leq \max( \lambda^{j-1},\lambda^{j})$}
\STATE{Add $\alpha_j^j$ to $\mathcal{A}$}
\ENDIF
\IF{$\min(\kappa^{j-1},\kappa^{j}) \leq \beta_j^j < \max( \lambda^{j-1},\lambda^{j})$}
\STATE{Add $\beta_j^j$ to $\mathcal{B}$}
\ENDIF

\IF {$\kappa^j > \lambda^{j-1}$}
\IF{$\kappa^{j-1} < \alpha_j^j$}
\STATE{$P := \bar{P}^{j-1}_L + l_j$; $Q := \bar{Q}^{j-1}_L$}
\ELSIF{$\alpha_j^j \leq \kappa^{j-1} < \beta_j^j$}
\STATE{$P := \bar{P}^{j-1}_L$; $Q := \bar{Q}^{j-1}_L + a_j$}
\ELSE
\STATE{$P := \bar{P}^{j-1}_L + u_j$;  $Q := \bar{Q}^{j-1}_L$}
\ENDIF
\STATE{Apply procedure {\sc LowerSubproblem}($j$)}
\ENDIF
\IF {$\lambda^j < \lambda^{j-1}$}
\IF{$\beta_j^j < \lambda^{j-1}$}
\STATE{$P:= \bar{P}^{j-1}_U + u_j$; $Q:= \bar{Q}^{j-1}_U$}
\ELSIF{$\alpha_j^j < \lambda^{j-1} \leq \beta_j^j$}
\STATE{$P := \bar{P}^{j-1}_U$; $Q := \bar{Q}^{j-1}_U + a_j$}
\ELSE
\STATE{$P := \bar{P}^{j-1}_U + l_j$; $Q := \bar{Q}^{j-1}_U$}
\ENDIF
\STATE{Apply procedure {\sc UpperSubproblem}($j$)}
\ENDIF

\end{algorithmic}
\end{multicols}
\end{algorithm}

The procedure {\sc SolveSubproblems}$(j)$ carries out the breakpoint search procedure for the subproblems QRAP$^j(L_j)$ and QRAP$^j(U_j)$ as described in Section~\ref{sec_BP_search} (Lines~38-59). This is done by first initializing the bookkeeping parameters for these breakpoint search procedures in Lines~39-46 and Lines~49-56 and subsequently applying the the procedures {\sc LowerSubproblem}$(j)$ (Line~47, Algorithm~\ref{alg_lower}) and {\sc UpperSubproblem}$(j)$ (Line~57, Algorithm~\ref{alg_upper}), which are identical in nature to Lines~5-22 of Algorithm~\ref{alg_QRAP_01}.  Before carrying out the breakpoint search procedure, two possible corner cases are considered in Lines~1-38 with regard to relation between the to-be-computed multipliers $\kappa^j$ and $\lambda^j$ and their predecessors $\kappa^{j-1}$ and $\lambda^{j-1}$. We briefly discuss these corner cases for $\kappa^j$; the corner cases for $\lambda^j$ are analogous.

\begin{algorithm}[ht!]
\caption{Procedure {\sc LowerSubproblem}($j$).}
\label{alg_lower}
\begin{multicols}{2}
\begin{algorithmic}[1]
\REPEAT
\STATE{Choose minimum to-be-considered breakpoint: $\delta := \max(\mathcal{A},\mathcal{B},\mathcal{K},\mathcal{L})$ and corresponding member set $\mathcal{D} \in \lbrace \mathcal{A},\mathcal{B},\mathcal{K},\mathcal{L} \rbrace$}
\IF{$P +  Q \delta = L_j$}
\STATE{$\kappa^j := \delta$; add $\kappa^j$ to $\mathcal{K}$}
\STATE{$\bar{P}^j_L := P$, $\bar{Q}^j_L := Q$}
\RETURN
\ELSIF{$P + Q \delta > L_j$}
\STATE{($\kappa^j < \delta$:) $\kappa^j := (L_j-P)/Q$; add $\kappa^j$ to $\mathcal{K}$}
\STATE{$\bar{P}^j_L := P$, $\bar{Q}^j_L := Q$}
\RETURN
\ELSE
\STATE{($\kappa^j > \delta$:) breakpoint~$\delta$ will be considered}
\IF{$\mathcal{D} \equiv \mathcal{A}$}
\STATE{Let breakpoint be $\delta \equiv \alpha_k^k$}
\STATE{$P := P - l_k$; $Q := Q +a_k$}
\STATE{Remove $\alpha_k^k$ from $\mathcal{A}$}
\ELSIF{$\mathcal{D} \equiv \mathcal{B}$}
\STATE{Let breakpoint be $\delta \equiv \beta_k^k$}
\STATE{$P := P + u_k$; $Q := Q - a_k$}
\STATE{Remove $\beta_k^k$ from $\mathcal{B}$}
\ELSIF{$\mathcal{D} \equiv \mathcal{K}$}
\STATE{Let breakpoint be $\delta \equiv \kappa^k$}
\STATE{$P := P - \bar{Q}^k_L \kappa^k$, $Q := Q + \bar{Q}^k_L$}
\STATE{Remove $\kappa^k$ from $\mathcal{K}$}
\ELSE
\STATE{Let breakpoint be $\delta \equiv \lambda^k$}
\STATE{$P := P + \bar{Q}^k_U \lambda^k$; $Q := Q - \bar{Q}^k_U$}
\STATE{Remove $\lambda^k$ from $\mathcal{L}$}
\ENDIF
\ENDIF
\UNTIL{$\kappa^j$ has been determined}
\end{algorithmic}
\end{multicols}
\end{algorithm}

\begin{algorithm}[ht!]
\caption{Procedure {\sc UpperSubproblem}($j$).}
\label{alg_upper}
\begin{multicols}{2}
\begin{algorithmic}[1]
\REPEAT
\STATE{Choose maximum to-be-considered breakpoint: $\delta := \min(\mathcal{A},\mathcal{B},\mathcal{K},\mathcal{L})$ and corresponding member set $\mathcal{D} \in \lbrace \mathcal{A},\mathcal{B},\mathcal{K},\mathcal{L} \rbrace$}
\IF{$P + Q \delta = U_j$}
\STATE{$\lambda^j := \delta$; add $\lambda^j$ to $\mathcal{L}$}
\STATE{$\bar{P}^j_U := P$, $\bar{Q}^j_U := Q$}
\RETURN
\ELSIF{$P + Q \delta < U_j$}
\STATE{($\lambda^j > \delta$:) $\lambda^j := (U_j - P)/Q)$; add $\lambda^j$ to $\mathcal{L}$}
\STATE{$\bar{P}^j_U := P$, $\bar{Q}^j_U := Q$}
\RETURN
\ELSE
\STATE{($\lambda^j < \delta$): breakpoint will be considered}
\IF{$\mathcal{D} \equiv \mathcal{A}$}
\STATE{Let breakpoint be $\delta \equiv \alpha_k^k$}
\STATE{$P := P + l_k$, $Q := Q - a_k$}
\STATE{Remove $\alpha_k^k$ from $\mathcal{A}$}
\ELSIF{$\mathcal{D} \equiv \mathcal{B}$}
\STATE{Let breakpoint be $\delta \equiv \beta_k^k$}
\STATE{$P := P - u_k$; $Q := Q + a_k$}
\STATE{Remove $\beta_k^k$ from $\mathcal{B}$}
\ELSIF{$\mathcal{D} \equiv \mathcal{K}$}
\STATE{Let breakpoint be $\delta \equiv \kappa^k$}
\STATE{$P := P + \bar{Q}^k_L \kappa^k$, $Q := Q - \bar{Q}^k_L$}
\STATE{Remove $\kappa^k$ from $\mathcal{K}$}
\ELSE
\STATE{Let breakpoint be $\delta \equiv \lambda^k$}
\STATE{$P := P - \bar{Q}^k_U \lambda^k$; $Q := Q + \bar{Q}^k_U $}
\STATE{Remove $\lambda^k$ from $\mathcal{L}$}
\ENDIF
\ENDIF
\UNTIL{$\lambda^j$ has been determined}
\end{algorithmic}
\end{multicols}
\end{algorithm}

The first corner case occurs when $\kappa^j = \kappa^{j-1}$ (Lines~1-9 in {\sc SolveSubproblems}$(j)$). This case corresponds to Lines~5-7 in Algorithm~\ref{alg_QRAP_01} , where the currently considered candidate multiplier~$\delta_i$ leads to a solution $x[\delta_i]$ that sums to $C$, i.e., $z[\delta_i] = C$. For QRAP-NC$^j(L_j)$, this case thus occurs if and only if $L_{j-1} + {x}_j^j[\kappa^j] = L_j$, i.e., if and only if ${x}_i^j (L_j) = {x}_i^{j-1} (L_{j-1})$ for all $i \in \mathcal{N}_{j-1}$ and $L_{j-1} + \max(l_j, \min(a_j \chi^j,u_j)) = L_j$. The second case (Lines~9-13) occurs when $\kappa^j < \kappa^{j-1}$ and corresponds to Lines~8-10 of Algorithm~\ref{alg_QRAP_01}, where the candidate multiplier~$\delta_i$ leads to a solution $x[\delta_i]$ whose sum is larger than $C$, i.e., $z[\delta_i] > C$. In QRAP-NC$^j(L_j)$, this case occurs if and only if $L_{j-1} + {x}_j^j[\kappa^j] > L_j$, i.e., if and only if ${x}_i^j (L_j) = {x}_i^{j-1} (L_{j-1})$ for all $i \in \mathcal{N}_{j-1}$ and $L_{j-1} + \max(l_j, \min(a_j \chi^j,u_j)) > L_j$. In both cases, it is not necessary to carry out the actual breakpoint search to find $\kappa^j$ since either $\kappa^j = \kappa^{j-1}$ (the first case) or $\kappa^j = (L_j - L_{j-1}) / a_j$ (the second case).

Whether or not one of the above mentioned corner cases occurs partly determines whether or not we have to include the new initial breakpoint values $\alpha_j^j$ and $\beta_j^j$ in the breakpoint search procedure. The algorithm makes this decision in Lines~33-38: $\alpha_j^j$ and $\beta_j^j$ are included only if they are in between the lowest and highest breakpoint values that can be considered in the breakpoint search. This lowest value is $\kappa^{j}$
 if $\kappa^j \leq \kappa^{j-1}$ (when one of the two corner cases for $\kappa^j$ occurs and thus this value has already been determined) and $\kappa^{j-1}$ otherwise (when breakpoint search is required to find $\kappa^j$). Analogously, the highest value is $\lambda^{j}$ if $\lambda^j \geq \lambda^{j-1}$ and $\lambda^{j-1}$ otherwise.

\subsubsection{Time complexity}
\label{sec_complexity}

We now establish the worst-case time complexity of Algorithm~\ref{alg_01} by means of the following lemma:
\begin{lemma}
Algorithm~\ref{alg_01} can be implemented such that its worst-case time complexity is $O(n \log n)$.
\label{lemma_time}
\end{lemma}
\begin{proof}
Observe that, throughout the algorithm and all its procedures, all operations have a total time complexity of $O(n)$ except for four operations on the sets $\mathcal{A}$, $\mathcal{B}$, $\mathcal{K}$, and $\mathcal{L}$ of to-be-considered breakpoints. For each of these breakpoint sets, say $\mathcal{D}$, these are finding the minimum and maximum breakpoint in $\mathcal{D}$ (Lines~2 and~18 in Algorithm~\ref{alg_BP} and Line~2 in Algorithms~\ref{alg_lower} and~\ref{alg_upper}), inserting a breakpoint value in $\mathcal{D}$ (Lines~13, 29, 34, and 37 in Algorithm~\ref{alg_BP}), and removing the minimum or maximum breakpoint from $\mathcal{D}$ (Lines~15 and 31 in Algorithm~\ref{alg_BP} and Lines~16, 20, 24, and 28 in Algorithms~\ref{alg_lower} and~\ref{alg_upper}). As we showed in Section~\ref{sec_opt_mult}, each breakpoint value is inserted and removed at most once during the course of the algorithm. Moreover, in the worst case, we have to find the minimum and maximum breakpoint value in $\mathcal{D}$ a number of $n$ times. Thus, the total number of breakpoint set operations is $O(n)$. If we maintain the breakpoint sets as min-max heaps \cite{Atkinson1986}, each of these operations can be executed in $O(1)$ (finding the minimum and maximum) and $O(\log n)$ (inserting and removing a breakpoint) time. This means that the total time complexity of all four breakpoint set operations is $O(n \log n)$ if we use min-max heaps to store the breakpoint sets. It follows that Algorithm~\ref{alg_01} can be implemented such that its worst-case time complexity is $O(n \log n)$.
\end{proof}
In practice, carrying out the breakpoint set operations might be faster if we use a different data structure than min-max heaps to maintain the breakpoint sets $\mathcal{A}$, $\mathcal{B}$, $\mathcal{K}$, and $\mathcal{L}$. For instance, when $n$ is small, simple arrays might be sufficient for fast insertion and removal of breakpoints, even though this increases the worst-case time complexity to $O(n^2)$. On the other hand, \cite{Hochbaum1995} suggests to keep the breakpoint sets by means of a so-called disjoint set data structure (see, e.g., \cite{Cormen2009}). Using such a structure, a sequence of $O(n)$ breakpoint insertions and deletions in sets of size at most $n$ can be done in $O(n)$ time using the algorithm in \cite{Gabow1985}. However, it is unclear whether the algorithm in \cite{Gabow1985} is fast in practice for two reasons. First, it is complicated and cumbersome to implement compared to other algorithms for insertion and removal operations on disjoint set data structures \cite{Galil1991}. Second, although the authors mention in the preliminary study \cite{Gabow1983} that their algorithm outperforms the then state-of-the-art, the literature contains hardly if any studies on its practical performance. Alternatively, one could use other algorithms (e.g., those evaluated in \cite{Patwary2010}) that have a worse worst-case time complexity but have been shown to be fast in practice.

\section{Evaluation}
\label{sec_eval}

In this section, we evaluate the performance of our Algorithm~\ref{alg_01} as presented in Section~\ref{sec_NC_alg}, to which we shall refer as \seq\ for clarity, and compare it with the state-of-the-art algorithms \infe\ from \cite{vdKlauw2017} and \dec\ from \cite{Vidal2018}. We carry out two types of experiments. First, we evaluate the performance of our algorithm on realistic instances of the battery scheduling problem BATTERY. For this, we tailor \seq\ to this problem and compare this implementation to a tailored implementation of \infe\ within the simulation tool DEMKit \cite{Hoogsteen2019}. Second, we compare the execution time and scalability of our algorithm and of \infe\ and \dec\ on synthetic instances with sizes ranging from $10$ to one million variables. We have implemented all three algorithms in Python (version 3.5) to be able to compare them to the implementation in DEMKit, which is also written in Python. All simulations and computations have been executed on a 2.60 GHz Dell Inspiron 15 with an Intel Core i7-6700HQ CPU and 16 GB of RAM.

 In Section~\ref{sec_instance}, we describe in more detail the problem instances that we use in the evaluation. Subsequently, in Section~\ref{sec_implement}, we discuss several implementation choices and in Section~\ref{sec_disc} we present and discuss the results of our evaluation.

\subsection{Problem instances}
\label{sec_instance}
For the comparison of the tailored implementation of our algorithm \seq\ with the tailored implementation of \infe\ within DEMKit, we generate realistic instances of the problem BATTERY. For this, we consider the setting where a battery charging schedule for two consecutive days needs to be computed. This scheduling horizon is divided into 15-minute time intervals, resulting in $n = 192$. To study the influence of the battery size on the solving time, we consider three scenarios that correspond to three different battery sizes and denote them by {\sc Small}, {\sc Medium}, and {\sc Large}. In these scenarios, the battery capacity is 20~kWh, 100~kWh, or 180~kWh and the (dis)charging rate is 4~kW, 20~kW, or 36~kW respectively. This leads to $\Delta t = \frac{1}{4}$ and to the values for $X_{\min}$, $X_{\max}$, and $D$ as given in Table~\ref{tab_par}. Note that this is equivalent to the situation where either 10, 50, or 90 percent of the households have installed a smaller ``home'' battery with a capacity of 5~kWh and a (dis)charging rate of 1~kW, which corresponds to real-life field tests such as described in \cite{Reijnders2018}. We set both the initial and target SoC to a given fraction of the capacity, i.e., $S_{\text{start}} = S_{\text{end}} = sD$, where $s \in \lbrace 0,0.1,0.2,\ldots,1 \rbrace$. For each scenario, we simulate 50 battery schedules of two days. As input for the base load $p$, we use measurement data of the actual power consumption of 40 households for 100 consecutive days that were obtained in the field test described in \cite{Hoogsteen2017b}.
\begin{table}[ht!]
\centering
\begin{tabular}{r | r r r }
\toprule
& $X_{\min}$ & $X_{\max}$ & $D$ \\
\midrule
{\sc Small} & $-4.0 \cdot 10^3$ & $4.0 \cdot 10^3$ & $8.0 \cdot 10^4$ \\
{\sc Medium} & $-2.0 \cdot 10^4$ & $2.0 \cdot 10^4$ & $4.0 \cdot 10^5$ \\
{\sc Large} & $-3.6 \cdot 10^4$ & $3.6 \cdot 10^4$ & $7.2 \cdot 10^5$ \\
\bottomrule
\end{tabular}
\caption{Parameter choices for the battery scheduling problem for each scenario.}
\label{tab_par}
\end{table}

For the scalability analysis, we generate synthetic instances in the same way as in \cite{Vidal2018}. For this, we consider instance sizes~$n$ in the set $\lbrace 10,20,50,100,200,500,\ldots,10^6 \rbrace$ and for each of these sizes, we generate 10 instances. In each instance, we sample the parameters $a$, $l$, and $u$ from the uniform distributions $U(0,1)$, $U(0.1,0.5)$, and $U(0.5,0.9)$ respectively. To generate the nested bounds $L$ and $U$, we first draw for each $i \in \mathcal{N}$ two values $X_i$ and $Y_i$ from the uniform distribution $U(l_i,u_i)$. Subsequently, we define for each $j \in \mathcal{N}$ the values $v_j := \sum_{i \in \mathcal{N}_j} X_i$ and $w_j :=  \sum_{i \in \mathcal{N}_j} Y_i$ and we set $L_j := \min (v_j, w_j)$ and $U_j := \max(v_j,w_j)$ for $j < n$ and $L_n = U_n = \frac{1}{2}(v_n + w_n)$.

\subsection{Implementation details}
\label{sec_implement}

In both the divide-and-conquer algorithm \dec\ and the infeasibility-guided algorithm \infe, we use Algorithm~\ref{alg_QRAP_01} to solve the QRAP subproblems. Note that using this algorithm instead of linear-time algorithms such as in \cite{Kiwiel2008a} increases the worst-case time complexity of \dec\ and \infe\ by a factor $O(\log n)$. However, for practically relevant problems sizes, this procedure is generally faster in practice and easier to implement than the linear-time algorithms in, e.g., \cite{Kiwiel2008a}.

For the double-ended queues needed in \seq\ for the optimal Lagrange multipliers $\kappa$ and $\lambda$, we use the Python container datatype {\tt deque}. Moreover, we initially implemented the double-ended priority queues for the lower and upper initial breakpoint values $(\alpha_i^i)_{i \in \mathcal{N}}$ and $(\beta_i^i)_{i \in \mathcal{N}}$ as symmetric min-max heaps \cite{Arvind1999}. However, initial tests indicated that using instead a coupled min-heap and max-heap implementation with total correspondence leads to similar or even lower execution times of the overall algorithm. Moreover, the latter data structure is much simpler to implement using the standard Python libary {\tt heapq}. Therefore, we use this method instead of min-max heaps. In this alternative method, we insert new breakpoints in both the min-heap and the max-heap and use the min-heap to find and delete a minimum breakpoint (in the lower subproblems) and the max-heap to find and delete a maximum breakpoint (in the upper subproblems). Moreover, we assign to each breakpoint a flag that is 1 if the breakpoint has been removed from either of the heaps and 0 otherwise. This prevents that we find a minimum (maximum) breakpoint in the min-heap (max-heap) that was already considered in the other heap and thus has been removed from the breakpoint search.

\subsection{Results and discussion}
\label{sec_disc}

In this section, we present and discuss the results of our evaluation. First, we discuss the results of the comparison of the tailored implementation of \seq\ with the tailored implementation of \infe\ within DEMKit. Figure~\ref{plot_DEM} shows the ratios between the execution times of the tailored implementation of \infe\ and that of \seq. Moreover, Tables~\ref{tab_result_small}-\ref{tab_result_large} contain for each scenario and each initial and target SoC value the mean, maximum, and coefficient of variation (CoV) of the execution times. The CoV is the sample deviation divided by the sample mean and is a suitable measure of the variation between samples when comparing different collections of samples with significantly different sample means.

\begin{figure}[ht!]
\centering
\begin{subfigure}[b]{.32\textwidth}
\centering
\includegraphics{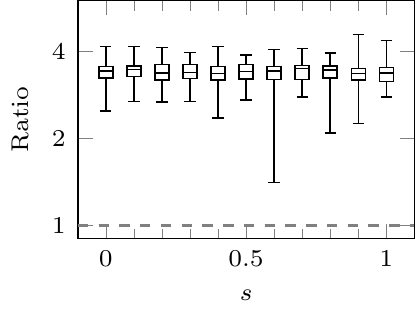}
\caption{{\sc Small}.}
\label{plot_DEM_small}
\end{subfigure}
\hspace{0.005\textwidth}%
\begin{subfigure}[b]{.32\textwidth}
\centering
\includegraphics{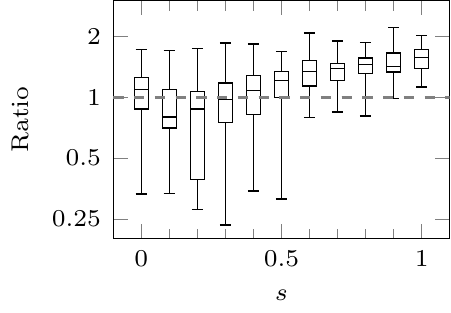}
\caption{{\sc Medium}.}
\label{plot_DEM_Medium}
\end{subfigure}
\hspace{0.005\textwidth}%
\begin{subfigure}[b]{.32\textwidth}
\centering
\includegraphics{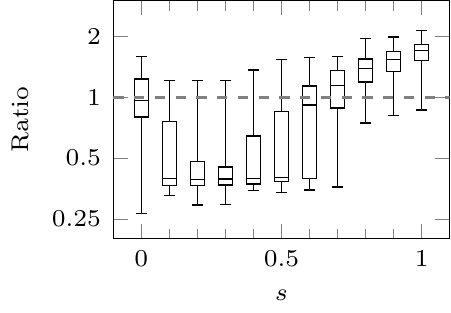}
\caption{{\sc Large}.}
\label{plot_DEM_small}
\end{subfigure}
\caption{Boxplots of the execution time of the tailored implementation of \infe\ within DEMKit divided by that of the tailored implementation of \seq\ for the three scenarios. Ratios larger than 1 imply that \seq\ was faster than \infe.}
\label{plot_DEM}
\end{figure}

\begin{table}[ht!]
\centering
\begin{tabular}{r | r r r | r r r}
\toprule
& \multicolumn{3}{c}{\seq} & \multicolumn{3}{|c}{\infe\ within DEMKit} \\
$s$
& Mean & Max & CoV & Mean & Max & CoV \\
\midrule
0 & $1.80 \cdot 10^{-3}$ & $2.09 \cdot 10^{-3}$ & $5.80 \cdot 10^{-2}$ & $6.15 \cdot 10^{-3}$ & $   7.23 \cdot 10^{-3}$ &  $7.61 \cdot 10^{-2}$ \\
 0.1&  $1.78 \cdot 10^{-3}$ & $   2.14 \cdot 10^{-3}$ &$   6.18 \cdot 10^{-2}$ & $   6.15 \cdot 10^{-3}$ & $   7.16 \cdot 10^{-3}$ & $   6.52 \cdot 10^{-2}$ \\
0.2&   $1.81 \cdot 10^{-3}$ & $   2.77 \cdot 10^{-8}$ & $   9.19 \cdot 10^{-2}$ & $   6.10 \cdot 10^{-3}$ & $   7.15 \cdot 10^{-3}$ & $   7.17 \cdot 10^{-2}$ \\
 0.3&  $1.79 \cdot 10^{-3}$ & $   2.25 \cdot 10^{-3}$ &  $   7.70 \cdot 10^{-2}$ & $   6.10 \cdot 10^{-3}$ & $   7.10 \cdot 10^{-3}$ & $   7.34 \cdot 10^{-2}$ \\
 0.4&  $1.79 \cdot 10^{-3}$ & $   2.44 \cdot 10^{-3}$ & $   7.25 \cdot 10^{-2}$ & $   6.07 \cdot 10^{-3}$ & $   7.11 \cdot 10^{-3}$ & $   7.10 \cdot 10^{-2}$ \\
  0.5& $1.77 \cdot 10^{-3}$ & $   2.25 \cdot 10^{-3}$ &  $   6.28 \cdot 10^{-2}$ & $   6.02 \cdot 10^{-3}$ & $   6.94 \cdot 10^{-3}$ &$   6.41 \cdot 10^{-2}$ \\
 0.6&  $1.83 \cdot 10^{-3}$ & $   4.00 \cdot 10^{-3}$ & $   1.89 \cdot 10^{-1}$ & $   6.01 \cdot 10^{-3}$ & $   7.09 \cdot 10^{-3}$ &  $   6.78 \cdot 10^{-2}$ \\
 0.7&  $1.77 \cdot 10^{-3}$ & $   2.21 \cdot 10^{-3}$ & $   6.49 \cdot 10^{-2}$ & $   6.05 \cdot 10^{-3}$ & $   6.97 \cdot 10^{-3}$ & $   7.69 \cdot 10^{-2}$ \\
 0.8&  $1.80 \cdot 10^{-3}$ & $   2.68 \cdot 10^{-3}$ &  $   1.00 \cdot 10^{-1}$ & $   6.05 \cdot 10^{-3}$ & $   7.06 \cdot 10^{-3}$ & $   7.52 \cdot 10^{-2}$ \\
 0.9&  $1.79 \cdot 10^{-3}$ & $   2.87 \cdot 10^{-3}$ &  $   1.24 \cdot 10^{-1}$ & $   5.97 \cdot 10^{-3}$ & $   7.88 \cdot 10^{-3}$ &  $   8.48 \cdot 10^{-2}$ \\
 1&  $1.79 \cdot 10^{-3}$ & $   2.18 \cdot 10^{-3}$ &  $   6.02 \cdot 10^{-2}$ & $   6.01 \cdot 10^{-3}$ & $   8.22 \cdot 10^{-3}$ &  $   8.81 \cdot 10^{-2}$ \\
\bottomrule
\end{tabular}
\caption{The mean, maximum, and coefficient of variation of the execution times of the tailored implementation of \seq\ and the tailored implementation of \infe\ within DEMKit for the scenario {\sc Small}.}
\label{tab_result_small}
\end{table}

\begin{table}[ht!]
\centering
\begin{tabular}{r | r r r | r r r}
\toprule
& \multicolumn{3}{c}{\seq} & \multicolumn{3}{|c}{\infe\ within DEMKit} \\
$s$
& Mean & Max & CoV & Mean & Max & CoV \\
\midrule
  0 & $ 1.71 \cdot 10^{-3} $ & $   2.40 \cdot 10^{-3} $ & $   1.19 \cdot 10^{-1} $ & $   1.88 \cdot 10^{-3} $ & $   3.45 \cdot 10^{-3} $ & $   3.10 \cdot 10^{-1} $ \\
0.1 & $   1.64 \cdot 10^{-3} $ & $   1.96 \cdot 10^{-3} $ & $   7.39 \cdot 10^{-2} $ & $   1.45 \cdot 10^{-3} $ & $   3.31 \cdot 10^{-3} $ & $   4.47 \cdot 10^{-1} $ \\
 0.2 & $  1.66 \cdot 10^{-3} $ & $   2.40 \cdot 10^{-3} $ & $   1.03 \cdot 10^{-1} $ & $   1.47 \cdot 10^{-3} $ & $   3.34 \cdot 10^{-3} $ & $   4.78 \cdot 10^{-1} $ \\
0.3 & $   1.76 \cdot 10^{-3} $ & $   3.50 \cdot 10^{-3} $ & $   1.85 \cdot 10^{-1} $ & $   1.68 \cdot 10^{-3} $ & $   3.43 \cdot 10^{-3} $ & $   4.41 \cdot 10^{-1} $ \\
0.4 & $   1.65 \cdot 10^{-3} $ & $   2.36 \cdot 10^{-3} $ & $   1.07 \cdot 10^{-1} $ & $   1.76 \cdot 10^{-3} $ & $   3.51 \cdot 10^{-3} $ & $   3.45 \cdot 10^{-1} $ \\
0.5 & $   1.66 \cdot 10^{-3} $ & $   2.42 \cdot 10^{-3} $ & $   1.08 \cdot 10^{-1} $ & $   1.95 \cdot 10^{-3} $ & $   3.27 \cdot 10^{-3} $ & $   2.83 \cdot 10^{-1} $\\
0.6 & $   1.65 \cdot 10^{-3} $ & $   2.48 \cdot 10^{-3} $ & $   9.75 \cdot 10^{-2} $ & $   2.23 \cdot 10^{-3} $ & $   3.52 \cdot 10^{-3} $ & $   2.26 \cdot 10^{-1} $ \\
0.7 & $   1.68 \cdot 10^{-3} $ & $   2.09 \cdot 10^{-3} $ & $   8.46 \cdot 10^{-2} $ & $   2.32 \cdot 10^{-3} $ & $   3.42 \cdot 10^{-3} $ & $   1.85 \cdot 10^{-1} $ \\
0.8 & $   1.64 \cdot 10^{-3} $ & $   1.94 \cdot 10^{-3} $ & $   6.79 \cdot 10^{-2} $ & $   2.36 \cdot 10^{-3} $ & $   3.19 \cdot 10^{-3} $ & $   1.61 \cdot 10^{-1} $ \\
 0.9 & $  1.70 \cdot 10^{-3} $ & $   2.17 \cdot 10^{-3} $ & $   1.02 \cdot 10^{-1} $ & $   2.53 \cdot 10^{-3} $ & $   3.50 \cdot 10^{-3} $ & $   1.63 \cdot 10^{-1} $ \\
1 & $   1.67 \cdot 10^{-3} $ & $   2.11 \cdot 10^{-3} $ & $   8.41 \cdot 10^{-2} $ & $   2.61 \cdot 10^{-3} $ & $   4.05 \cdot 10^{-3} $ & $   1.54 \cdot 10^{-1} $  \\
   \bottomrule
\end{tabular}
\caption{The mean, maximum, and coefficient of variation of the execution times of the tailored implementation of \seq\ and the tailored implementation of \infe\ within DEMKit for the scenario {\sc Medium}.}
\label{tab_result_medium}
\end{table}

\begin{table}[ht!]
\centering
\begin{tabular}{r | r r r | r r r}
\toprule
& \multicolumn{3}{c}{\seq} & \multicolumn{3}{|c}{\infe\ within DEMKit} \\
$s$
& Mean & Max & CoV & Mean & Max & CoV \\
\midrule
0 & $   1.59 \cdot 10^{-3} $ & $   2.15 \cdot 10^{-3} $ & $   7.33 \cdot 10^{-2} $ & $   1.60 \cdot 10^{-3} $ & $   2.76 \cdot 10^{-3} $ & $   2.98 \cdot 10^{-1} $ \\
 0.1 & $  1.58 \cdot 10^{-3} $ & $   1.90 \cdot 10^{-3} $ & $   7.33 \cdot 10^{-2} $ & $   8.93 \cdot 10^{-4} $ & $   2.11 \cdot 10^{-3} $ & $   5.22 \cdot 10^{-1} $ \\
0.2 & $   1.58 \cdot 10^{-3} $ & $   1.94 \cdot 10^{-3} $ & $   8.61 \cdot 10^{-2} $ & $   7.82 \cdot 10^{-4} $ & $   2.04 \cdot 10^{-3} $ & $   4.77 \cdot 10^{-1} $ \\
0.3 & $   1.55 \cdot 10^{-3} $ & $   2.16 \cdot 10^{-3} $ & $   8.79 \cdot 10^{-2} $ & $   7.45 \cdot 10^{-4} $ & $   2.00 \cdot 10^{-3} $ & $   4.66 \cdot 10^{-1} $ \\
 0.4 & $  1.54 \cdot 10^{-3} $ & $   2.05 \cdot 10^{-3} $ & $   7.06 \cdot 10^{-2} $ & $   8.22 \cdot 10^{-4} $ & $   2.30 \cdot 10^{-3} $ & $   5.18 \cdot 10^{-1} $ \\
 0.5 & $  1.54 \cdot 10^{-3} $ & $   2.05 \cdot 10^{-3} $ & $   7.67 \cdot 10^{-2} $ & $   1.02 \cdot 10^{-3} $ & $   2.42 \cdot 10^{-3} $ & $   5.51 \cdot 10^{-1} $\\
 0.6 & $  1.53 \cdot 10^{-3} $ & $   1.83 \cdot 10^{-3} $ & $   6.24 \cdot 10^{-2} $ & $   1.34 \cdot 10^{-3} $ & $   2.63 \cdot 10^{-3} $ & $   4.38 \cdot 10^{-1} $ \\
0.7 & $   1.55 \cdot 10^{-3} $ & $   1.77 \cdot 10^{-3} $ & $   6.22 \cdot 10^{-2} $ & $   1.72 \cdot 10^{-3} $ & $   2.61 \cdot 10^{-3} $ & $   2.75 \cdot 10^{-1} $ \\
0.8 & $   1.53 \cdot 10^{-3} $ & $   1.86 \cdot 10^{-3} $ & $   5.80 \cdot 10^{-2} $ & $   2.09 \cdot 10^{-3} $ & $   2.95 \cdot 10^{-3} $ & $   1.87 \cdot 10^{-1} $ \\
0.9 & $   1.53 \cdot 10^{-3} $ & $   1.94 \cdot 10^{-3} $ & $   6.34 \cdot 10^{-2} $ & $   2.31 \cdot 10^{-3} $ & $   2.98 \cdot 10^{-3} $ & $   1.69 \cdot 10^{-1} $ \\
1 & $   1.52 \cdot 10^{-3} $ & $   1.77 \cdot 10^{-3} $ & $   4.58 \cdot 10^{-2} $ & $   2.51 \cdot 10^{-3} $ & $   3.16 \cdot 10^{-3} $ & $   1.36 \cdot 10^{-1} $ \\
   \bottomrule
\end{tabular}
\caption{The mean, maximum, and coefficient of variation of the execution times of the tailored implementation of \seq\ and the tailored implementation of \infe\ within DEMKit for the scenario {\sc Large}.}
\label{tab_result_large}
\end{table}

Tables~\ref{tab_result_small}-\ref{tab_result_large} show that the mean execution time of \seq\ is similar in each scenario, whereas that of \infe\ appears to decrease as the battery size increases. This implies that also the ratios between the execution times decrease as the battery size increases, which is confirmed by the boxplots in Figure~\ref{plot_DEM}. In particular, a smaller battery size seems to imply that \seq\ is likely to be faster than \infe\, whereas \infe\ is likely to be faster for larger battery size. The reason for this is that the execution time of \infe\ heavily depends on the number of tight nested constraints in an optimal solution. To support this fact, we plot in Figure~\ref{plot_DEM_tight} boxplots of these numbers. Note that when the initial SoC is 20\% or 30\% of the battery capacity in the scenario {\sc Large}, in only 4 of the 50 instances the number of tight constraints was more than 1, meaning that in the remaining 46 instances the optimal solution to the relaxation of the problem did not violate any of the nested constraints. The relation between the number of tight nested constraints and the ratios is also strongly visible when comparing Figures~\ref{plot_DEM} and~\ref{plot_DEM_tight}: the ratios increase as the number of tight constraints increases.

\begin{figure}[ht!]
\centering
\begin{subfigure}[b]{.32\textwidth}
\centering
\includegraphics{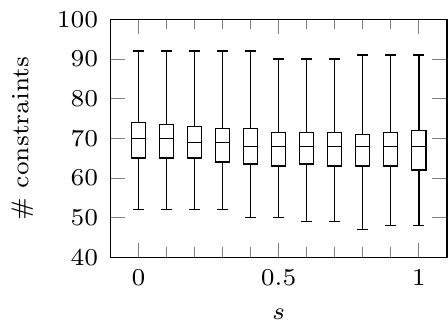}
\caption{{\sc Small}.}
\label{plot_DEM_tight_small}
\end{subfigure}
\hspace{0.005\textwidth}%
\begin{subfigure}[b]{.32\textwidth}
\centering
\includegraphics{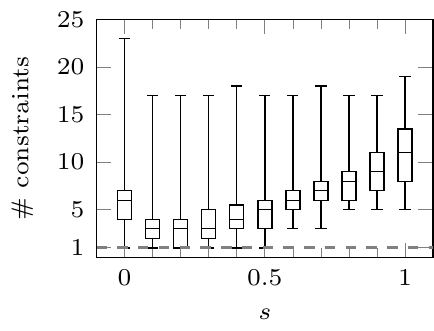}
\caption{{\sc Medium}.}
\label{plot_DEM_tight_Medium}
\end{subfigure}
\hspace{0.005\textwidth}%
\begin{subfigure}[b]{.32\textwidth}
\centering
\includegraphics{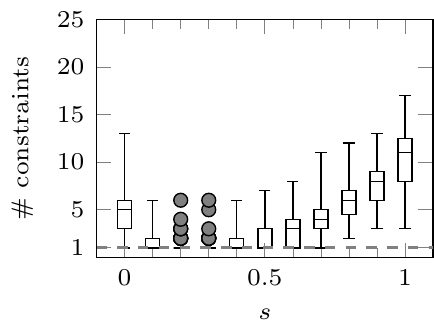}
\caption{{\sc Large}.}
\label{plot_DEM_tight_large}
\end{subfigure}
\caption{Boxplots of the number of tight constraints in the optimal solutions for the three scenarios.}
\label{plot_DEM_tight}
\end{figure}

From these results, we can derive a ``rule of thumb'' for the choice of a proper algorithm to use given the expected number of tight nested constraints. To this end, we compute for each number of tight constraints the percentage of instances where the tailored implementation of \seq\ runs faster than the tailored implementation of \infe\ within DEMKit given the optimal solution has this particular number of tight nested constraints (see Table~\ref{tab_win}). These values suggest that when the number of tight constraints is more than $\frac{4}{192} \approx 2.1$ percent, our algorithm is faster in more than 50\% of the instances.  In particular, when the number of tight constraints is $\frac{7}{192} \approx 3.6$ percent or more, the tailored implementation of our algorithm \seq\ is always faster. 

\begin{table}[ht!]
\centering
\begin{tabular}{r | r r r r r r r}
\toprule
Number of tight nested constraints & 1 & 2 & 3 & 4 & 5 & 6 & $\geq 7$ \\
\midrule
``Win'' percentage & 0.0 & 2.2 & 30.2 & 62.8 & 83.8 & 93.4 & 100 \\
\bottomrule
\end{tabular}
\caption{Percentage of instances where the tailored implementation of \seq\ is faster than the tailored implementation of \infe\ within DEMKit given the number of tight nested constraints in their optimal solutions.}
\label{tab_win}
\end{table}

Note that this rule-of-thumb is in line with the physical interpretation of tight nested constraints in BATTERY. For this, note that a battery being completely empty or full is equivalent to a nested constraint of BATTERY being tight. When the charging rates of the battery are large, the battery is better able to, at a given moment, flatten large peaks or drops in power consumption. However, the latter is also dependent on whether there is enough space (energy) left in the battery to store (dispatch) this energy, which is more likely when the battery capacity is large. Thus, when adopting a large battery for load profile flattening, it is less likely that it will be completely empty or full.

Although the ratio between the execution times of \seq\ and \infe\ appears to depend significantly on the battery size, the maximum and CoV of the execution times of \seq\ is on average around 1.9 and 3.0 times smaller than that of \infe\ respectively. This means that the execution times of \seq\ are on average more stable than those of \infe, regardless of the battery size. For DEM in general and DEMKit in particular, this is beneficial since the coordination and optimization of schedules for different devices is often done in parallel due to the decentralized nature of the coordination (see, e.g., \cite{Hoogsteen2018}). As a consequence, the execution time of the entire coordination and optimization framework is constrained by the maximum execution time required for solving one (subset of) device-level optimization problem(s). Thus, using \seq\ instead of \infe\ within such a framework may significantly reduce the overall execution time of the framework. 

In the following, we present and discuss the results of the scalability evaluation. Figure~\ref{plot_compare} shows the execution times of the three algorithms \seq, \infe, and \dec, and Table~\ref{tab_scale} shows for each studied instance size~$n$ the mean and CoV of the execution times of the corresponding instances. The added regression lines in Figure~\ref{plot_compare} are the fitted power laws of the execution times, i.e., for each algorithm we fit the function $\phi(n) = c_1 \cdot n^{c_2}$ to the execution times. These lines indicate that the practical execution time of both \seq\ and \dec\ is close to $O(n)$ and that of \infe\ is actually slightly less than $O(n)$. Thus, it can be expected that for very large values of $n$, more precisely for $n >2.90 \cdot 10^9$, \infe\ is on average faster than \seq. However, the CoV for \infe\ are around one order of magnitude larger than those of both \seq\ and \dec. This suggests that the execution time of the latter two algorithms is significantly less affected by the choice of problem parameters than \infe. This is in line the results of the comparison of the tailored implementation of \seq\ for BATTERY with that of \infe\ within DEMKit.

\begin{figure}[ht!]
\centering
\includegraphics{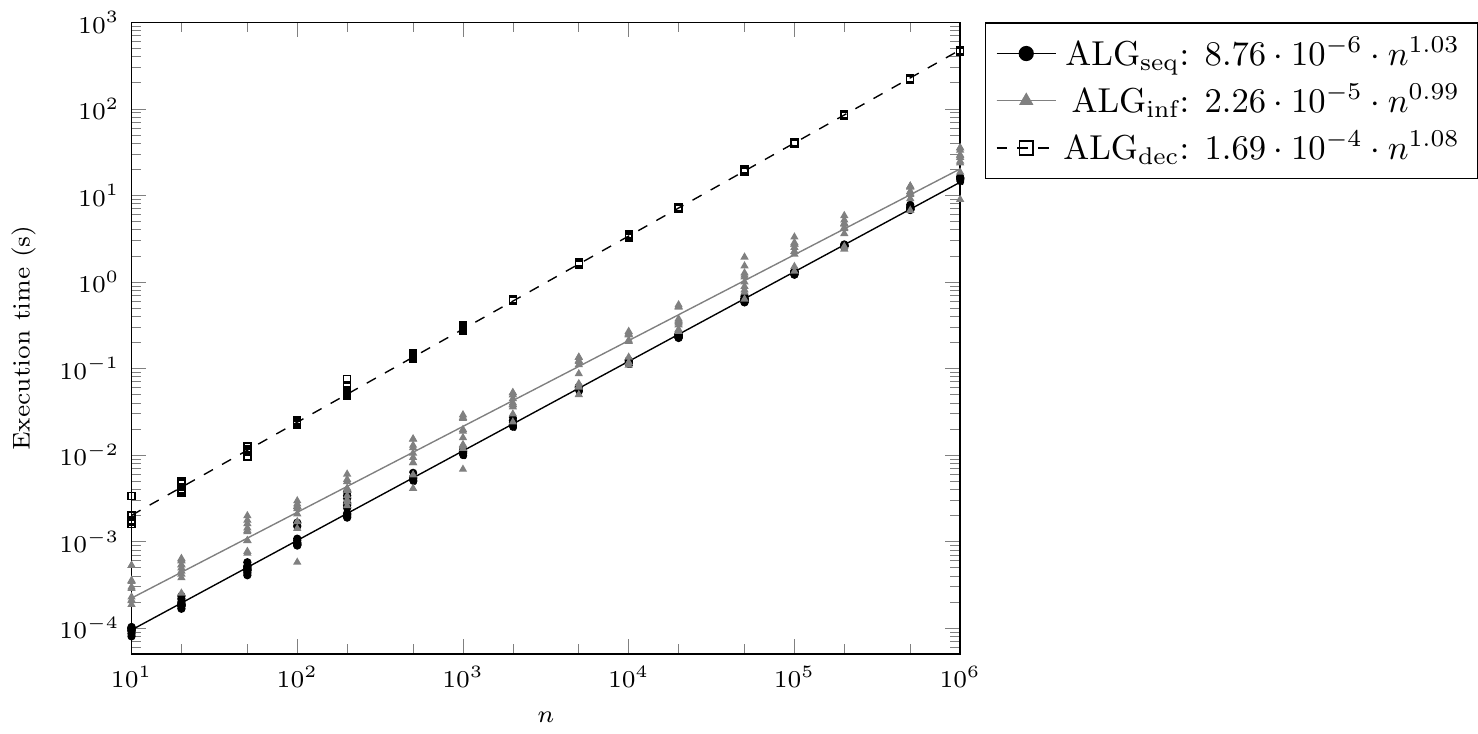}
\caption{Execution times of \seq\ (circles, black), \infe\ (triangles, gray), and \dec\ (squares, open).}
\label{plot_compare}
\end{figure}

\begin{table}[ht!]
\centering
\begin{tabular}{r | lll |lll}
\toprule
& \multicolumn{3}{c}{Mean} & \multicolumn{3}{|c}{CoV} \\
$n$ & \seq & \infe & \dec & \seq & \infe & \dec \\
\midrule
10 &$   9.47 \cdot 10^{-5}$ & $   3.17 \cdot 10^{-4}$ & $   1.90 \cdot 10^{-3}$ & $   7.40 \cdot 10^{-2}$ & $   3.94 \cdot 10^{-1}$ & $   2.74 \cdot 10^{-1}$  \\
20 & $   1.92 \cdot 10^{-4}$ & $   4.92 \cdot 10^{-4}$ & $   4.17 \cdot 10^{-3}$ & $   1.03 \cdot 10^{-1}$ & $   2.40 \cdot 10^{-1}$ & $   1.10 \cdot 10^{-1}$  \\
 50 & $  4.82 \cdot 10^{-4}$ & $   1.30 \cdot 10^{-3}$ & $   1.14 \cdot 10^{-2}$ & $   9.69 \cdot 10^{-2}$ & $   3.21 \cdot 10^{-1}$ & $   8.19 \cdot 10^{-2}$  \\
 100 & $  1.08 \cdot 10^{-3}$ & $   2.14 \cdot 10^{-3}$ & $   2.37 \cdot 10^{-2}$ & $   2.46 \cdot 10^{-1}$ & $   3.37 \cdot 10^{-1}$ & $   3.88 \cdot 10^{-2}$  \\
 200 & $  2.62 \cdot 10^{-3}$ & $   4.01 \cdot 10^{-3}$ & $   5.57 \cdot 10^{-2}$ & $   2.45 \cdot 10^{-1}$ & $   2.75 \cdot 10^{-1}$ & $   1.58 \cdot 10^{-1}$  \\
500 & $   5.29 \cdot 10^{-3}$ & $   1.02 \cdot 10^{-2}$ & $   1.39 \cdot 10^{-1}$ & $   7.11 \cdot 10^{-2}$ & $   3.70 \cdot 10^{-1}$ & $   4.51 \cdot 10^{-2}$  \\
 1,000 & $  1.06 \cdot 10^{-2}$ & $   1.82 \cdot 10^{-2}$ & $   2.89 \cdot 10^{-1}$ & $   4.61 \cdot 10^{-2}$ & $   4.09 \cdot 10^{-1}$ & $   4.56 \cdot 10^{-2}$  \\
 2,000 & $  2.29 \cdot 10^{-2}$ & $   4.05 \cdot 10^{-2}$ & $   6.11 \cdot 10^{-1}$ & $   8.10 \cdot 10^{-2}$ & $   2.33 \cdot 10^{-1}$ & $   1.79 \cdot 10^{-2}$  \\
5,000 & $   5.82 \cdot 10^{-2}$ & $   9.49 \cdot 10^{-2}$ & $   1.63$ & $   3.83 \cdot 10^{-2}$ & $   3.50 \cdot 10^{-1}$ & $   3.10 \cdot 10^{-2}$  \\
 10,000 & $  1.170 \cdot 10^{-1}$ & $   1.91 \cdot 10^{-1}$ & $   3.41$ & $   4.05 \cdot 10^{-2}$ & $   3.38 \cdot 10^{-1}$ & $   2.50 \cdot 10^{-2}$  \\
20,000 & $   2.37 \cdot 10^{-1}$ & $   3.85 \cdot 10^{-1}$ & $   7.18$ & $   3.75 \cdot 10^{-2}$ & $   2.63 \cdot 10^{-1}$ & $   1.60 \cdot 10^{-2}$  \\
 50,000 & $  6.19 \cdot 10^{-1}$ & $   1.12$ & $   1.94 \cdot 10^{01}$ & $   4.93 \cdot 10^{-2}$ & $   3.53 \cdot 10^{-1}$ & $   2.24 \cdot 10^{-2}$  \\
100,000 & $   1.28$ & $   2.35$ & $   4.03 \cdot 10^{01}$ & $   3.20 \cdot 10^{-2}$ & $   2.57 \cdot 10^{-1}$ & $   1.56 \cdot 10^{-2}$  \\
 200,000 & $  2.61$ & $   4.21$ & $   8.40 \cdot 10^{01}$ & $   1.67 \cdot 10^{-2}$ & $   2.59 \cdot 10^{-1}$ & $   1.35 \cdot 10^{-2}$  \\
 500,000 & $  7.07$ & $   1.09 \cdot 10^{01}$ & $   2.22 \cdot 10^{02}$ & $   4.11 \cdot 10^{-2}$ & $   1.76 \cdot 10^{-1}$ & $   1.11 \cdot 10^{-2}$  \\
 1,000,000 & $  1.56 \cdot 10^{01}$ & $   2.66 \cdot 10^{01}$ & $   4.66 \cdot 10^{02}$ & $   3.32 \cdot 10^{-2}$ & $   3.07 \cdot 10^{-1}$ & $   1.29 \cdot 10^{-2}$  \\
   \bottomrule
\end{tabular}
\caption{Mean and coefficient of variation of the execution times.}
\label{tab_scale}
\end{table}

The results in Figure~\ref{plot_compare} and Table~\ref{tab_scale} indicate that on average \seq\ is 27.2 times faster than \dec\ and 1.95 times faster than \infe. With regard to the performance of \dec, we acknowledge that \dec\ and in particular the updating scheme for the single-variable bounds can probably be implemented more efficiently than in the current implementation. To reduce the influence of the overall implementation on the results of this study, we measured the total time that is spent in \dec\ on solving QRAP subproblems and compared this to the execution times of \seq\ and \infe. This alternative time represents the time that is minimally required to solve all QRAP subproblems regardless of the implementation of the scheme used to update the single-variable bounds. These measurements indicate that on average around 59\% of the total execution time of \dec\ is spent on solving QRAP subproblems. However, this time is still on average 15.9 times more than the execution time of \seq\ and 9.5 times more than that of \infe.

\section{Conclusions}
\label{sec_concl}

We proposed an $O(n \log n)$ time algorithm for quadratic resource allocation problems with lower and upper bound constraints on nested sums of variables. As opposed to existing algorithms with the same time complexity, our algorithm can achieve the $O(n \log n)$ time complexity using only basic data structures and is therefore easier to implement. In computational experiments, we demonstrate the good practical performance of our approach, both on synthetic data and on realistic instances from the application area of decentralized energy management (DEM) for smart grids.

Our approach builds upon monotonicity arguments that find their origin in the validity of greedy algorithms for convex optimization problems over polymatroids \cite{Hochbaum1994,Hochbaum1995}. Such monotonicity arguments have been primarily studied for resource allocation problems where the objective function is separable, i.e., can be written as the sum of single-variable functions. However, in previous work \cite{SchootUiterkamp2019a} we prove the validity of similar monotonicity arguments to solve a \emph{nonseparable} resource allocation problem with so-called generalized bound constraints. Moreover, recent results on the use of interior-point methods for nested resource allocation problems \cite{Slager2019,Wright2020} suggest that incorporating specific nonseparable terms in the objective function does not increase the complexity of the used solution method.  Thus, one interesting direction for future research is to investigate whether one can use monotonicity arguments to derive efficient algorithms for resource allocation problems over nested constraints with nonseparable objective functions. 

With regard to the application within DEM systems, we compared our algorithm with an existing implementation of the state-of-the-art algorithm of \cite{vdKlauw2017} within a simulation tool for DEM research. One of our objectives was to decide which of these two algorithm is more suitable to use for a given (type of) problem instance. It would be worthwhile to conduct a more thorough comparison and to develop an automated procedure to decide which algorithm is most likely to be faster. Moreover, the nonseparable version of the studied problem mentioned in the previous paragraph is related to energy management of batteries in three-phase distribution networks, where load profile flattening on all three phases together is required to avoid blackouts in these networks \cite{Weckx2015,Hoogsteen2017b,SchootUiterkamp2019a}. Thus, research in this direction is also relevant in the context of DEM.

\appendix

\section{Proofs of Lemmas~\ref{lemma_mono},~\ref{lemma_chi}, and~\ref{lemma_xn}}

\subsection{Proof of Lemma~\ref{lemma_mono}}
\label{pf_mono}
\begin{lemma_mono}
If $L_j \leq A \leq B \leq U_j$, we have ${x}^j(A) \leq {x}^j(B)$ for a given $j \in \mathcal{N}$.
\end{lemma_mono}
\begin{proof}
For convenience, we include the equality constraint~(\ref{sub_02}) into the nested constraints~(\ref{sub_03}) by replacing these nested constraints by
\begin{equation*}
\tilde{L}_k \leq \sum_{i \in \mathcal{N}_k} x_i \leq \tilde{U}_k, \quad k \in \mathcal{N}_j,
\end{equation*}
where $\tilde{L}_k = L_k$ and $\tilde{U}_k = U_k$ for $k<j$, and $\tilde{L}_j = \tilde{U}_j = C$. The Karush-Kuhn-Tucker (KKT) optimality conditions (see, e.g., \cite{Boyd2004}) for the subproblem QRAP-NC$^j(C)$ are as follows:
\begin{subequations}
\begin{align}
 \frac{x_i}{a_i} + \sum_{k=i}^j (\eta^j_k - \zeta^j_k) + \mu^j_i - \nu^j_i =0, & \quad i \in \mathcal{N}_j, \label{KKT_01} \\
\tilde{L}_k \leq \sum_{k \in \mathcal{N}_i} x_i \leq \tilde{U}_k, &   \quad k\in \mathcal{N}_j, \label{KKT_02} \\
\eta^j_i \left( \tilde{U}_i - \sum_{k \in \mathcal{N}_i} x_k \right) =0, & \quad i\in \mathcal{N}_j, \label{KKT_03} \\
\zeta^j_i \left( \sum_{k \in \mathcal{N}_i} x_k - \tilde{L}_i \right) = 0,& \quad i \in \mathcal{N}_j, \label{KKT_04} \\
\mu^j_i (u_i - x_i)  =0,& \quad i\in \mathcal{N}_j, \label{KKT_05} \\
\nu^j_i (x_i - l_i)  =0,& \quad i\in \mathcal{N}_j, \label{KKT_06} \\
\eta^j_i, \zeta^j_i, \mu^j_i, \nu^j_i  \geq 0,& \quad i\in \mathcal{N}_j. \label{KKT_07}
\end{align} \label{KKT}%
\end{subequations}
Let $(\zeta^j(C),\eta^j(C),\mu^j(C),\nu^j(C))$ denote the Lagrange multipliers corresponding to the optimal solution ${x}^j(C)$. Thus, $({x}^j(C),\zeta^j(C),\eta^j(C),\mu^j(C),\nu^j(C))$ satisfy the KKT-conditions~(\ref{KKT}).

Suppose that there exists an index $s \in \mathcal{N}$ such that ${x}_s^j(A) > {x}_s^j(B)$. Let $r$ be the largest index with $r \leq s$ such that $\sum_{k \in \mathcal{N}_{r-1}} {x}_k^j(A) \geq \sum_{k \in \mathcal{N}_{r-1}} {x}_k^j(B)$, and let $t$ be the smallest index with $t \geq s$ such that $\sum_{k \in \mathcal{N}_t} {x}_k^j(A) \leq \sum_{k \in \mathcal{N}_t} {x}_k^j(B)$. By definition of $r$, $s$, and $t$, we have that
\begin{equation*}
\sum_{i=r}^t {x}_i^j(B) = \sum_{i \in \mathcal{N}_t} {x}_i^j(B) - \sum_{i \in \mathcal{N}_{r-1}} {x}_i^j(B)
\geq  \sum_{i \in \mathcal{N}_t} {x}_i^j(A) - \sum_{i \in \mathcal{N}_{r-1}} {x}_i^j(A)
= \sum_{i=r}^t {x}_i^j(A). 
\end{equation*}
Moreover, observe that we cannot have $r=s=t$ simultaneously. Indeed, if $r=s=t$, then we have by definition of $r$, $s$, and $t$ that
\begin{equation*}
\sum_{k \in \mathcal{N}_s} {x}_k^j(A) \leq \sum_{k \in \mathcal{N}_s} {x}_k^j(B) \leq \sum_{k \in \mathcal{N}_{s-1}} {x}_k^j(A) + {x}_s^j(B).
\end{equation*}
This implies ${x}_s^j(A) \leq {x}_s^j(B)$, which is a contradiction. Thus, either $r<s$ or $s<t$ or both. 

We show that we obtain a contradiction if $r<s$. The proof for the case where $s<t$ is symmetrical. If $r<s$, the following holds:
\begin{itemize}
\item
By definition of $r$ and $s$, we have
\begin{equation*}
\sum_{k \in \mathcal{N}_r} x_k^j(A) < \sum_{k \in \mathcal{N}_r} x_k^j(B) = \sum_{k \in \mathcal{N}_{r-1}} x_k^j(B) + x_r^j(B) \leq \sum_{k \in \mathcal{N}_{r-1}} x_k^j(A) + x_r^j(B).
\end{equation*}
Thus, $x_r^j(A) < x_r^j(B)$.
\item
For each $k$ such that $r \leq k \leq s-1$, we have by definition of $r$ and $s$ and KKT-condition~(\ref{KKT_02}) that
\begin{equation*}
\tilde{L}_k \leq \sum_{i \in \mathcal{N}_k} {x}_i^j(A) < \sum_{i \in \mathcal{N}_k} {x}_i^j(B) \leq \tilde{U}_k.
\end{equation*}
It follows from KKT-conditions~(\ref{KKT_03}),~(\ref{KKT_04}), and (\ref{KKT_07}) that $\zeta_k^j(B) = \eta_k^j(A) = 0$. Thus, for each $r \leq k \leq s-1$, we have
\begin{equation*}
\sum_{i=k}^j (\eta_i^j(A)-\zeta_i^j(A)) - \sum_{i=k+1}^j (\eta_i^j(A)-\zeta_i^j(A))
= \eta_k^j(A)-\zeta_k^j(A)
\leq 0
\end{equation*}
and
\begin{equation*}
\sum_{i=k}^j (\eta_i^j(B)-\zeta_i^j(B)) - \sum_{i=k+1}^j (\eta_i^j(B)-\zeta_i^j(B))
= \eta_k^j(B)-\zeta_k^j(B)
\geq 0.
\end{equation*}
In particular, this implies that
\begin{equation}
\sum_{i=r}^j (\eta_i^j(A)-\zeta_i^j(A)) \leq \sum_{i=s}^j (\eta_i^j(A)-\zeta_i^j(A))
\label{eq_klA}
\end{equation}
and
\begin{equation}
\sum_{i=r}^j (\eta_i^j(B)-\zeta_i^j(B)) \geq \sum_{i=s}^j (\eta_i^j(B)-\zeta_i^j(B)).
\label{eq_klB}
\end{equation}
\item
We have $l_r \leq x_r^j(A) < x_r^j(B) \leq u_r$. It follows from KKT-conditions~(\ref{KKT_05})-(\ref{KKT_07}) that
\begin{equation}
\nu_r^j(B) = \mu_r^j(A) = 0. \label{eq_nmr}
\end{equation}
Similarly, since $l_s \leq x_s^j(B) < x_s^j(A) \leq u_s$, we have by KKT-conditions~(\ref{KKT_05})-(\ref{KKT_07}) that
\begin{equation}
\nu_s^j(A) = \mu_s^j(B) = 0. \label{eq_nms}
\end{equation}
\end{itemize}
We can now derive a contradiction as follows:
\begin{subequations}
\begin{align}
\sum_{i=s}^j (\eta_i^j(A) - \zeta_i^j(A)) &= - \frac{{x}_s^j(A)}{a_s}  -\mu_s^j(A) + \nu_s^j(A) \label{cont_01} \\
& < -\frac{{x}_s^j(B)}{a_s}  -\mu_s^j(B) + \nu_s^j(B) \label{cont_02}\\
&=\sum_{i=s}^j (\eta_i^j(B) - \zeta_i^j(B)) \label{cont_03}\\
& \leq \sum_{i=r}^j (\eta_i^j(B) - \zeta_i^j(B)) \label{cont_04}\\
&=-  \frac{{x}_r^j(B)}{a_r}  -\mu_r^j(B) + \nu_r^j(B) \label{cont_05} \\
&<-  \frac{{x}_r^j(A)}{a_r}  -\mu_r^j(A) + \nu_r^j(A) \label{cont_06}\\
&=\sum_{i=r}^j (\eta_i^j(A) - \zeta_i^j(A)) \label{cont_07}\\
&\leq \sum_{i=s}^j (\eta_i^j(A) - \zeta_i^j(A)). \label{cont_08}
\end{align}
\end{subequations}
Here,
\begin{itemize}
\item
(\ref{cont_01}),~(\ref{cont_03}),~(\ref{cont_05}), and~(\ref{cont_07}) follow from KKT-condition~(\ref{KKT_01});
 \item
(\ref{cont_02}) follows from Equation~(\ref{eq_nms}) and the fact that ${x}_s^j(A) > {x}_s^j(B)$ and $a_s > 0$;
 \item
(\ref{cont_04}) follows from Equation~(\ref{eq_klB});
\item
(\ref{cont_06}) follows from Equation~(\ref{eq_nmr}) and the fact that ${x}_r^j(A) < {x}_r^j(B)$ and $a_s > 0$;
\item
(\ref{cont_08}) follows from Equation~(\ref{eq_klA}).
\end{itemize}
It follows that $x_s^j(A) \leq x_s^j(B)$.
\end{proof}

\subsection{Proof of Lemma~\ref{lemma_chi}}
\label{pf_chi}
\begin{lemma_chi}
We have $\chi^n = \kappa^n = \lambda^n$. Moreover, for each $j \in \mathcal{N}_{n-1}$, we have:
\begin{enumerate}
\item $ \chi^{j+1} \leq \kappa^j$ implies $\sum_{i \in \mathcal{N}_j} {x}_i^n (R) = L_j$ and $\chi^j = \kappa^j$;
\item $\lambda^j \leq \chi^{j+1}$ implies $\sum_{i \in \mathcal{N}_j} {x}_i^n (R) = U_j$ and $\chi^j = \lambda^j$,
\item $\kappa^j < \chi^{j+1} < \lambda^{j}$ implies $L_j < \sum_{i \in \mathcal{N}_j} {x}_i^n (R) < U_j$ and $\chi^j = \chi^{j+1}$.
\end{enumerate}
\end{lemma_chi}
\begin{proof}
We have $\chi^n = \kappa^n = \lambda^n$ since we defined $L_n = U_n = R$ and by definition of the solution ${x}^n(R)$ the nested constraints $L_n \leq \sum_{i \in \mathcal{N}} {x}_i^n(L_n)$ and $\sum_{i \in \mathcal{N}} {x}_i^n(U_n) \leq U_n$ are tight. We prove the lemma by considering each of its three cases separately for each $j < n$:

\begin{enumerate}
\item
We prove this part of the lemma for the case that $j$ is the largest index smaller than ${\ell}_{j+1}$ such that $\chi^{j+1} \leq \kappa^j$, i.e., $\chi^{k+1} > \kappa^k$ for all $k \in \lbrace j +1, \ldots, \ell_{j+1} - 1 \rbrace$. Using this result, we show as follows that the other case, i.e., both the situations where either $j = \ell_{j+1}$ or where there exists an index $k > j$ that it is the largest index in the set $ \lbrace j+1,\ldots, \ell_{j+1} - 1 \rbrace$ such that $\chi^{k+1} \leq \kappa^k$, leads to a contradiction. In the former situation, it follows that $j+1 > j = \ell_{j+1} \geq j+1$, which is a contradiction. In the latter situation, the lemma applies for $k$, meaning that $\sum_{i \in \mathcal{N}_k} {x}_i^n (R) = L_k$ and thus $\ell_{k} = k$. However, we also have by definition of $\ell_{j+1}$ that $\ell_{k} = \ell_{j+1}$ since $j+1 \leq k < \ell_{j+1}$. This implies $k = \ell_{j+1}$, which is a contradiction. 

If $\chi^{j+1} \leq \kappa^{j}$, it follows from the lower breakpoint relations in Equation~(\ref{eq_rel_lower_BP}) that we have either $\alpha_i^{j+1} \geq \kappa^{j} \geq \chi^{j+1}$ (if $\kappa^{j} < \beta_i^{j}$) or $\alpha_i^{j+1} = \beta_i^{j} \leq \kappa^{j} \leq \lambda^{j}$ (if $\beta_i^{j} \leq \kappa^{j}$) for all $i \leq j + 1$. We show that in both cases it holds that ${x}_i^{{\ell}_{j+1}}(V_{{\ell}_{j+1}}) = {x}_i^{j}(L_{j})$:
\begin{itemize}
\item
In the former case, note that $\alpha_i^{k} \leq \alpha_i^{k+1}$ for all $k < n$ by Equation~(\ref{eq_rel_lower_BP}) and that $\chi^{j+1} = \chi^{k}$ for all $k \in \lbrace j+1,\ldots,\ell_{j+1} \rbrace$ by definition of $\chi^{j+1}$. Since $\chi^{k+1} > \kappa^k$ for all $k \in \lbrace j+1,\ldots,\ell_{j+1} - 1 \rbrace$, we have that $\alpha_i^k \geq \alpha_i^{j+1} \geq \chi^{j+1} = \chi^{k+1} > \kappa^k$ for all $k \in \lbrace j+1, \ldots, \ell_{j+1} - 1 \rbrace$. Thus, ${x}_i^{k}(L_{k}) = \bar{l}_i^{k} = {x}_i^{k-1}(L_{k-1})$ for all $k \in \lbrace j+1, \ldots, \ell_{j+1} - 1 \rbrace$, which implies that ${x}_i^j(L_j) = {x}_i^{\ell_{j+1} - 1} (L_{\ell_{j+1} - 1})$. Moreover, note that since $\alpha^{\ell_{j+1}} \geq \alpha^{j+1} \geq \chi^{j+1} = \chi^{\ell_{j+1}}$, we have that ${x}_i^{\ell_{j+1}}(V_{\ell_{j+1}}) = {x}_i^{\ell_{j+1} - 1}(L_{\ell_{j+1} - 1})$. It follows that ${x}_i^{\ell_{j+1}}(V_{\ell_{j+1}}) = {x}_i^{j}(L_{j})$.

\item
The latter case implies that ${x}_i^{j}(L_{j}) = {x}_i^{j}(U_{j}) = \bar{u}_i^{j}$. It follows by Lemmas~\ref{lemma_mono} and \ref{lemma_bound} that ${x}_i^{{\ell}_{j+1}}(V_{{\ell}_{j+1}}) \leq {x}_i^{{\ell}_{j+1}} (U_{{\ell}_{j+1}}) \leq {x}_i^{j} (U_{j}) = {x}_i^{j}(L_{j}) \leq {x}_i^{{\ell}_{j+1}}(L_{{\ell}_{j+1}}) \leq {x}_i^{{\ell}_{j+1}}(V_{{\ell}_{j+1}})$.
\end{itemize}
On the one hand, if $V_{{\ell}_{j+1}} = L_{{\ell}_{j+1}}$, we have
\begin{align*}
L_{j} &= \sum_{i \mathcal{N}_j} {x}_i^{j}(L_{j}) = \sum_{i \in \mathcal{N}_j} {x}_i^{{\ell}_{j+1}}(L_{{\ell}_{j+1}})
= L_{{\ell}_{j+1}} - \sum_{i=j+1}^{{\ell}_{j+1}} {x}_i^{{\ell}_{j+1}}(L_{{\ell}_{j+1}})
\geq
\sum_{i \in \mathcal{N}_{{\ell}_{j+1}}} {x}_i^n(R)
-\sum_{i=j+1}^{{\ell}_{j+1}} {x}_i^{n}(R) \\
&= \sum_{i \in \mathcal{N}_j} {x}_i^{n}(R)
\geq L_{j},
\end{align*}
where the inequality follows since $\sum_{i \in \mathcal{N}_{\ell_{j+1}}} {x}_i^n(R) = L_{\ell_{j+1}}$ and by Lemma~\ref{lemma_bound}. On the other hand, if $V_{{\ell}_{j+1}} = U_{{\ell}_{j+1}}$, we have by Lemma~\ref{lemma_bound} that
\begin{equation*}
L_{j} = \sum_{i \in \mathcal{N}_j} {x}_i^{j}(L_{j}) = \sum_{i \in \mathcal{N}_j} {x}_i^{{\ell}_{j+1}}(U_{{\ell}_{j+1}})
\geq \sum_{i \in \mathcal{N}_j} {x}_i^n(U_n)
\geq L_{j}.
\end{equation*}
In both cases, it follows that $\sum_{i \in \mathcal{N}_j} {x}_i^n (R) = L_{j}$, from which it follows directly that $\chi^{j} = \kappa^{j}$.

\item
The proof for the case $\lambda^{j} \geq \chi^{j+1}$ is analogous to the proof for the case $\chi^{j+1} \leq \kappa^{j}$.

\item

Suppose that ${x}_i^{j} (L_{j}) = {x}_i^n(L_n)$ holds for all $i < j+1$. By Lemma~\ref{lemma_bound}, this implies that ${x}_i^{k}(L_{k}) = {x}_i^{j} (L_{j}) = {x}_i^n (L_n)$ for all $k \in \lbrace j, \ldots,  n \rbrace$ and $i < j+1$. In particular, we have that ${x}_i^{k} (L_{k}) = \bar{l}_i^{k}$ for all $k \in \lbrace j+1,\ldots,n \rbrace$, which implies that $\kappa^{k} \leq \alpha_i^{k}$. Furthermore, note that for any $k' \in \mathcal{N}$ there is at least one index $i_{k'} \leq k$ such that $\alpha_{i_{k'}}^{k'} \leq \kappa^{k'} < \beta_{i_{k'}}^{k'}$. Otherwise, there exists $\epsilon > 0$ such that $\kappa^{k'} + \epsilon$ is an optimal Lagrange multiplier. It follows from the relation between $\alpha_{i_{k'}}^{k'}$ and $\alpha_{i_{k'}}^{k'+1}$ in Equation~(\ref{eq_rel_lower_BP}) that $\alpha_{i_{k'}}^{k'+1} = \kappa^{k'}$ for any $k' < n$. This implies in particular that $\alpha_{i_{k}}^{k+1} = \kappa^{k} \leq \alpha_{i_{k-1}}^{k}$ for all $k \in \lbrace j+1, \ldots, n \rbrace$. It follows that $\kappa^{{\ell}_{j+1}} \leq \alpha_{i_{j}}^{j+1} = \kappa^{j}$ and thus that $\kappa^{{\ell}_j+1} < \chi^{j+1} = \chi^{{\ell}_{j+1}}$. Since $\chi^{{\ell}_{j+1}} \in \lbrace \kappa^{{\ell}_{j+1}}, \lambda^{{\ell}_{j+1}} \rbrace$, we have $\chi^{{\ell}_{j+1}} = \lambda^{{\ell}_{j+1}}$, from which it follows that $\sum_{i \in \mathcal{N}_{{\ell}_{j+1}}} {x}_i^{n} (R) = U_{{\ell}_{j+1}}$. However, this implies that
\begin{equation*}
\sum_{i \in \mathcal{N}_{{\ell}_{j+1}}} {x}_i^{{\ell}_{j+1}} (L_{{\ell}_{j+1}})
=
\sum_{i \in \mathcal{N}_{{\ell}_{j+1}}} {x}_i^{{\ell}_{j+1}} (R)
=
U_{{\ell}_{j+1}}
\geq
\sum_{i \in \mathcal{N}_{{\ell}_{j+1}}} {x}_i^{{\ell}_{j+1}} (U_{{\ell}_{j+1}})
\geq
\sum_{i \in \mathcal{N}_{{\ell}_{j+1}}} {x}_i^{{\ell}_{j+1}} (L_{{\ell}_{j+1}}).
\end{equation*}
This implies that $\sum_{i \in \mathcal{N}_{{\ell}_{j+1}}} {x}_i^{{\ell}_{j+1}} (L_{{\ell}_{j+1}})  = \sum_{i \in \mathcal{N}_{{\ell}_{j+1}}} {x}_i^{{\ell}_{j+1}} (U_{{\ell}_{j+1}})$, from which it follows that $L_{{\ell}_{j+1}} = U_{{\ell}_{j+1}}$ by the monotonicity of ${x}^{\ell_{j+1}}(\cdot)$ as proven in Lemma~\ref{lemma_mono}. However, this is a contradiction with the assumption that $L_{k} < U_{k}$ for all $k < n$. Hence, there must be at least one index $i'$ such that ${x}_{i'}^{j} (L_{j})< {x}_{i'}^n (R)$. It follows that $L_{j} = \sum_{i \in \mathcal{N}_j} {x}_i^{j}(L_{j}) < \sum_{i \in \mathcal{N}_j} {x}_i^n (R)$.

To prove that $\sum_{i \in \mathcal{N}_j} {x}_i^n(R) < U_{j}$, we can use a similar argument wherein we show that the proposition ${x}_i^{j}(U_{j}) = {x}_i^n(R)$ cannot be true for all $i < n$. Together, this implies that $L_{j} < \sum_{i \in \mathcal{N}_j} {x}_i^n(R) < U_{j}$, from which it follows directly that $\chi^{j} = \chi^{{\ell}_{j+1}} = \chi^{j+1}$.
\end{enumerate}
\end{proof}

\subsection{Proof of Lemma~\ref{lemma_xn}}
\label{pf_xn}

\begin{lemma_xn}
For each $i \in \mathcal{N}$, we have
\begin{equation*}
{x}_i^n (R) =
\begin{cases}
l_i & \text{if } \chi^i < \alpha_i^i, \\
a_i \chi^i & \text{if } \alpha_i^i \leq \chi^i < \beta_i^i, \\
u_i & \text{if } \beta_i^i \leq \chi^i.
\end{cases}
\end{equation*}
\end{lemma_xn}%

\begin{proof} 
Let $\mathcal{J}$ denote the set of indices whose corresponding nested lower or upper constraint is tight in ${x}^n (R)$. More precisely,
\begin{equation*}
\mathcal{J}
:= \lbrace k_j \ | \ j \in \mathcal{N} \rbrace \equiv \lbrace j_1,\ldots,j_q \rbrace,
\end{equation*}
where $q := |\mathcal{J}|$ and $j_1 < \cdots < j_q$. For a given $p \in \lbrace 1,\ldots, q \rbrace$, note that since either the lower or upper nested constraint corresponding to $j_p$ is tight in the solution ${x}^n(R)$, we have that $\sum_{i \in \mathcal{N}_{j_p}} {x}_i^n(R) = V_{j_p}$. This implies that the vector $({x}_i^n(R))_{1 \leq i \leq j_p}$ is the optimal solution to the subproblem QRAP-NC$^{j_p}(V_{j_p})$, i.e., to the problem
\begin{align}
\text{QRAP-NC}^{j_p}(V_{j_p}) \ : \ 
\min_{x \in \mathbb{R}^{j_p}} \
 &\sum_{i \in \mathcal{N}_{j_p}} \frac{1}{2} \frac{x_i^2}{a_i}\nonumber \\
\text{s.t. } & \sum_{i \in \mathcal{N}_{j_p}} x_i = V_{j_p}, \nonumber \\
& L_k \leq \sum_{i \in \mathcal{N}_k} x_i \leq U_k, \quad k \in \lbrace 1,\ldots, j_p - 1 \rbrace, \label{QRAP_NC_02_03} \\
& l_i \leq x_i \leq u_i, \quad i \in \lbrace 1,\ldots,j_p \rbrace. \nonumber
\end{align}
Note that in the optimal solution $({x}_i^n(R))_{ i \in \mathcal{N}_{j_p}}$ to this problem, none of the nested constraints~(\ref{QRAP_NC_02_03}) for $k$ with $j_{p-1} < k < j_p$ are tight. As a consequence, when deriving the reformulated equivalent problem QRAP$^{j_p}(V_{j_p})$, it follows from Lemmas~\ref{lemma_bound} and~\ref{lemma_add} that we may replace the single-variable bounds~(\ref{sub_alt_03}) for $i$ with $j_{p-1} < i < j_p$ by the original variable bounds $l_i \leq x_i \leq u_i$. Thus, we can reformulate QPRAP-NC$^{j_p} (V_{j_p})$ to
\begin{align*}
\text{QRAP}^{j_p}(V_{j_p}) \ : \ 
\min_{x \in \mathbb{R}^j} \
&\sum_{i \in \mathcal{N}_{j_p}} \frac{1}{2} \frac{x_i^2}{a_i} \\
\text{s.t. } & \sum_{i \in \mathcal{N}_{j_p}} x_i = V_{j_p},  \\
& {x}_i^{j_p-1}(L_{j_p-1}) \leq x_i \leq {x}_i^{j_p-1}(U_{j_p-1}), \quad i \in \lbrace 1,\ldots, j_{p-1} \rbrace,  \\
&l_j \leq x_j \leq u_j, \quad i \in \lbrace j_{p-1}+1,\ldots,j_p \rbrace.
\end{align*} 
Recall that $\chi^{j_p}$ is the optimal Lagrange multiplier for this problem. As a consequence, we can directly compute ${x}_i^{j_p} (R)$ for $i \in \lbrace j_{p-1}+1,\ldots,j_p \rbrace$ using Equation~(\ref{eq_opt_QRA}):
\begin{equation*}
{x}_i^n (R) = {x}_i^{j_p} (V_{j_p})= \begin{cases}
l_i & \text{if } \chi^{j_p} < \alpha_i^i, \\
a_i \chi^{j_p} & \text{if } \alpha_i^i \leq \chi^{j_p} < \beta_i^i, \\
u_i & \text{if } \beta_i^i \leq \chi^{j_p}.
\end{cases}
\end{equation*}
The result of the lemma follows since we have $\chi^{j_p} = \chi^i $ for each $i \in \lbrace j_{p-1},\ldots, j_p \rbrace$ by definition of $j_p$ and $\chi^i$.
\end{proof}

\section*{Acknowledgment}
This research has been conducted within the SIMPS project (647.002.003) supported by NWO and Eneco.

\bibliographystyle{abbrv}

\bibliography{library_nested}

\end{document}